\documentclass[12pt]{amsart}
\usepackage{latexsym,amssymb,amsmath,amscd} 

\usepackage{tikz,enumerate}
\usepackage{pb-diagram, pb-xy}
\usepackage[all]{xy}

\usepackage[centering,margin=1in]{geometry}
\usepackage[numbers]{natbib}

\theoremstyle{plain}
 \newtheorem{theorem}{Theorem}[section]
 
  \newtheorem{lemma}[theorem]{Lemma}
  \newtheorem{corollary}[theorem]{Corollary}

 \newtheorem{definition}[theorem]{Definition}
  \newtheorem{example}[theorem]{Example}

  \newtheorem{remark}[theorem]{Remark}
	 \newtheorem{remarks}[theorem]{Remarks}

\def\ZZ{{\mathbb Z}}

\def\RR{{\mathbb R}}

\def \integers{\mathbb{Z}}
\def \reals{\mathbb{R}}

\def \Bzero{\mathbf{0}}
\def \g{\mathfrak{g}}
\def \wt{\mathrm{wt}}
\renewcommand{\l}[1] {
l^{J}_{#1}
}
\newcommand{\inner}[2]{
\langle #1, #2 \rangle 
}
\def \A{\mathcal{F}}
\def \sgn{\mathrm{sgn}}
\newcommand{\casetwo}[4]{\left\{ \begin{array}{ll} #1 &\mbox{if $#2$} \\[2mm] #3 &\mbox{if $#4$}\,. \end{array} \right.}
\newcommand{\floor}[1]{\lfloor #1 \rfloor}

\newlength{\cellsize}
\newcommand\tableau[1]{
\vcenter{
\let\\=\cr
\baselineskip=-16000pt
\lineskiplimit=16000pt
\lineskip=0pt
\halign{&\tableaucell{##}\cr#1\crcr}}}

\newcommand{\tableaucell}[1]{{%
\def \arg{#1}\def \void{}%

\ifx \void \arg
\vbox to \cellsize{\vfil \hrule width \cellsize height 0pt}%
\else
\unitlength=\cellsize
\begin{picture}(1,1)
\put(0,0){\makebox(1,1){$#1$}}
\put(0,0){\line(1,0){1}}
\put(0,1){\line(1,0){1}}
\put(0,0){\line(0,1){1}}
\put(1,0){\line(0,1){1}}
\end{picture}%
\fi}}
\begin{document}

\title
{From the weak Bruhat order to crystal posets}
\author{Patricia Hersh}
\email{plhersh@ncsu.edu}

\author{Cristian Lenart}
\email{clenart@albany.edu}

\thanks{
The first author was funded by NSF grants  DMS-1200730 and DMS-1500987.  The second author was funded by NSF grants DMS-1101264 and  DMS-1362627.
}


\begin{abstract}
We investigate the ways in which fundamental properties of the weak Bruhat order on a Weyl group can be lifted (or not) to a corresponding highest weight crystal graph, viewed as a partially ordered set; the latter projects to the weak order via the 
key map.  First, a crystal theoretic analogue of the statement that any two reduced expressions for the same Coxeter group element are related by Coxeter moves is proven for all lower intervals $[\hat{0},v]$ in a simply or doubly laced crystal.  On the other hand, it is shown that no finite set of moves exists, even in type $A$, for arbitrary crystal graph intervals.  In fact, it is shown that there are relations of arbitrarily high degree amongst crystal operators that are not implied by lower degree relations.   
Second, for crystals associated to  Kac-Moody algebras  it is shown for lower intervals that 
the M\"obius function is always $0$ or $\pm 1$, and in finite type this is also proven  for upper 
intervals, with a precise formula given  in each case. 
Moreover, the order complex for each of these  intervals is proven  to be homotopy equivalent to a  ball or to a sphere of some dimension, despite often not being shellable.
For  general intervals, examples are constructed  
with arbitrarily large M\"obius function, again even in type $A$.  Any interval having M\"obius function other than $0$ or $\pm 1$ is shown to contain within it a relation amongst crystal operators that is not implied by the relations giving rise to the local structure of the crystal, making precise a tight  relationship between M\"obius function and these somewhat unexpected relations appearing in crystals.  

New properties of the key map are also derived.  The key is shown to be determined entirely by the edge-colored poset-theoretic structure of the crystal, and a recursive algorithm is given for calculating it.  In finite types, the fiber of the longest element of any parabolic subgroup of the Weyl group is also proven to have a unique minimal and a unique maximal element; this property fails for more general elements of the Weyl group. 
\end{abstract}

\maketitle

\color{black}
\section{Introduction}
\label{definition-section}
\color{black}

Crystal graphs, introduced by Kashiwara \cite{Kas0}, are a powerful tool in the representation theory of symmetrizable Kac-Moody algebras and their quantum algebras. They are colored directed graphs encoding important data regarding quantum group representations when the quantum parameter goes to $0$. In this paper, a crystal always means one associated with an integrable highest weight representation of a symmetrizable Kac-Moody algebra (more precisely, of the corresponding quantum algebra). Several combinatorial models for crystals exist, ranging from uniform models, for all Lie types (such as the Littelmann path model, the alcove model, the Lusztig and string parametrizations of the canonical basis), to specialized models which only work in type $A$ (semistandard Young tableaux) or classical types. We will use both types of models in this paper, as well as a more abstract viewpoint due to Stembridge for simply laced crystals. 

Highest weight crystals are, in fact, partially ordered sets (posets), with cover relations $u\lessdot v$ 
whenever  there exists $i$ such that $v = f_i(u)$ for $f_i$ a crystal operator (defined shortly).  It is useful  to color such a  cover relation $i$, thereby giving the crystal  the structure of an edge-colored poset.   The resulting crystal posets project to the left weak  Bruhat order on the corresponding Weyl group via the so-called key map (cf. \cite{LS,Li,Reiner-Shimozono}). Previous interest in this poset map stems from the fact that it detects the smallest Demazure subcrystal containing a given crystal vertex. 

One main 
theme in 
our work is to show how  many well-known properties of the weak Bruhat order can be lifted to the crystal poset via the key.  However, there is a dichotomy in crystals in that we show that the properties under consideration  hold for lower intervals $[\hat{0},v]$ (and hence in finite type also for upper intervals $[v,\hat{1}]$), whereas we also construct  intervals  $[u,v]$ that fail to have these same properties, in fact  failing in arbitrarily extreme ways (as explained later). This dichotomy 
does not exist in the 
 weak Bruhat order, 
since  any interval $[u,v]$  in the weak Bruhat order  is isomorphic to a lower interval $[e,vu^{-1}]$.   
 
This work  led  us to  develop  new properties of the key along the way. First, in Section~\ref{sec:keyalg}, we give an efficient recursive algorithm to calculate the key.  In the process, we demonstrate that the key is determined  entirely by the edge-colored  poset-theoretic structure of the crystal graph in a uniform way. This algorithm works in the level of generality of symmetrizable Kac-Moody algebras. Second, we prove in  Theorem~\ref{unique-min} that the subposet of a 
 crystal poset mapping under the key to the longest element $w_\circ (J)$ of any parabolic subgroup $W_J$ has  unique minimal and maximal elements; we also give examples showing that this property fails for arbitrary Weyl group elements.  



A fundamental (and quite useful)  property of  Coxeter groups $W$ is the 
so-called ``word property'':
 any two reduced expressions $r_1,r_2$  for an element $w\in W$  are connected by a series of braid moves.  In other words, $r_2$ may be obtained from $r_1$ by  a series of steps each either (a) replacing some  $s_is_j$ by $s_js_i$, in the case that  $(s_is_j)^2 = e$, or (b) more generally replacing the alternation $s_is_js_i  \dots $ of two letters of length $m(i,j)$ by the alternation $s_js_is_j\dots   $ again of length $m(i,j)$, assuming that  $(s_is_j)^{m(i,j)} = e$ and that $m(i,j)$ is as small as possible with this property.
The word property  appears e.g. as Theorem~3.3.1 in \cite{BB},  and for type $A$ also appears as Theorem~1.1.2 in \cite{Garsia}.   We will want to use the following  equivalent, poset-theoretic formulation of the word property: for any $u\le v$ in weak Bruhat order, all saturated chains from $u$ to $v$ are connected by braid moves (see Proposition 3.1.2 in \cite{BB}), using the natural identification of saturated chains from $u$ to $v$ with reduced expressions for $vu^{-1}$ that is obtained  by placing the label $s_i$ on  $u\lessdot s_i u$ and reading off the sequence of labels encountered in proceeding up a saturated chain.  Another fundamental property of weak Bruhat order is that its M\"obius function $\mu(u,v)$  only takes values $0,\pm 1$, and in fact each open interval $(u,v)$ has order complex homotopy equivalent to a ball or a sphere.  

We will develop analogues of both of these properties of weak order  for crystal graphs,  in Theorem~\ref{connect-moves} and Corollary~\ref{Moebius-theorem}, respectively, but only for lower intervals (along with upper intervals in finite types).  The role of braid relations is played in the context of crystals of simply laced type  by the local relations 
$f_i f_j (u) = f_j f_i (u)$ and $f_i f_j f_j f_i (v) = f_j f_i f_i f_j (v)$ amongst crystal operators developed by John Stembridge.  
 The connectedness result, Theorem ~\ref{connect-moves}  in the simply laced case, 
holds for simply and doubly laced crystals. 
The topological result, Theorem ~\ref{homotopy-proof},  holds for all crystals  coming from 
symmetrizable Kac-Moody algebras.  
One might then  expect such results for arbitrary crystal  poset 
intervals $[u,v]$.  However,  we give examples in type $A$, in Section~\ref{negative-section},
  showing that the corresponding statements are false in general, and in fact fail to arbitrarily extreme degrees.    
 {We also exhibit  in  Example~\ref{non-lattice}  that type $A$ crystal posets are not always lattices. }

For crystals arising from symmetrizable Kac-Moody algebras (including in finite type), 
we prove in Corollary  ~\ref{Moebius-theorem}  that  $\mu (\hat{0},v) $ only takes values $0,\pm 1$, and  likewise for  $\mu (v,\hat{1})$ when a unique maximal element $\hat{1}$  exists within the crystal.  This follows from a stronger homotopy-theoretic statement proven in Theorem ~\ref{homotopy-proof}, by the interpretation of the M\"obius function as reduced Euler characteristic.  
More specifically, we deduce that $\mu (\hat{0},v) = 0$ unless the key of $v$ is the longest element  $w_\circ (J)$ of a parabolic subgroup $W_J$ and $v$ is the minimal element of the crystal having key $w_\circ (J)$,  in which  case we have $\mu (\hat{0},v) = (-1)^{|J|}$.  {On the other hand, we construct in Theorem~\ref{big-Moebius} intervals $[u,v]$ in type $A$ with $\mu (u,v)$ arbitrarily large in value.}  Similarly  to the weak Bruhat order (cf. \cite{BB}, \cite{Bj2},\cite{Edelman},\cite{Edelman-Walker}), these crystal posets are typically not shellable (a powerful topological condition developed extensively in \cite{Bjorner}, \cite{BW}), even when one restricts attention to the topologically well-understood  lower (or upper) intervals.  


  Our work leads  to new information about the structure of crystals coming from highest weight  representations, especially in the simply laced case.  Our results imply that the M\"obius function can be an efficient tool for detecting intervals in a crystal, regarded as a poset, where there must be relations amongst crystal operators that are not generated by the known set of  relations that one might expect to generate all relations.    First, we use M\"obius functions to discover  relations in Section ~\ref{negative-section}  that are not generated by those  of the form $f_i f_j (u) = f_j f_i (u) $ and $f_i f_j f_j f_i (u) = f_j f_i f_i f_j f_i (u)$  governing the local structure developed by Stembridge in  \cite{Stembridge}. 
  We had initially believed  that the M\"obius function would be $0, \pm 1$ everywhere, but  found counterexamples to this by computer search.  Upon closer examination, we  realized that these same examples also showed that these crystal posets were not always lattices and that not all relations among the  crystal operators are generated by the ones giving rise to  Stembridge's local structure.  Eventually we realized that  it  was not a coincidence that the same poset intervals  had both of these somewhat surprising features.
 Indeed we prove in Section ~\ref{SB-labeling}, specifically in Theorem ~\ref{moebius-yields-crystal-relations},  that  any interval with M\"obius function taking a value other than $0$ or $\pm 1$  must include within it  a relation amongst crystal operators not implied by Stembridge's local relations.  
  The main tool in developing this connection from M\"obius functions to these seemingly  unexpected relations amongst crystal operators was the theory of $SB$-labelings from \cite{HM}.

Now let us discuss briefly another type of motivation for this work.  The 
topological structure for fundamental posets arising in other parts of algebra, 
for instance finite group theory and commutative algebra,  
have played an important role there in characterizing important properties such as solvability and supersolvability in finite  group theory, and various  
properties of free resolutions including bounds on Betti numbers  in commutative algebra
(cf. \cite{Bjorner}, \cite{Eagon-Reiner}, \cite{HRW}, \cite{Miller-Sturmfels},  \cite{Peeva}, \cite{Shareshian}, \cite{Sh-W}, \cite{Smith}, \cite{St-supersolvable}).  It is natural to ask in other algebraic settings, such as the present one, 
whether topological properties of posets may again be useful as a tool for obtaining characterization results.
Corollary~\ref{crystal-type} shows that the crystals of highest weight $\rho$ (where, in type $A$, the weight $\rho$ is just the staircase partition) are special, in the sense that they have a particularly close structure to that of the weak order on the Weyl group, unlike the other crystals. 
{Corollary~\ref{Moebius-theorem} will  
directly  imply  that the M\"obius function alone can detect whether a given crystal element is the minimal element whose key is the longest element of a parabolic subgroup $W_J$.
This also yields a simple method to calculate the key of such elements $u$ directly, i.e. without needing to know the key of lower elements in the crystal poset.

Before proving the main  new results of this paper in Sections \ref{sec:keyalg}-\color{black} 
\ref{SB-labeling}, \color{black} 
we first provide background material,  which is organized as follows.  Sections~\ref{sec1} and  
~\ref{local-section} review general properties of crystals and their local structure, while Section~\ref{poset-topology-bg} provides background on partially ordered sets and on topological combinatorics.  These  sections,  together  with a suitable definition of  the key map (either taken from  Section~\ref{sectkey} or from Section~\ref{sec:keyalg}), will suffice to understand the proofs of our positive results regarding M\"obius functions, homotopy type, and connectedness of the set of saturated chains in  lower  and upper intervals under Stembridge moves.  Section~\ref{ssyt} reviews the SSYT model for type $A$ crystals, so as  to prepare  readers to understand our infinite families of examples giving negative results in Section~\ref{negative-section}, regarding M\"obius functions and connectedness for arbitrary type $A$ crystal poset intervals.  Section~\ref{alcmod} explains the more general alcove model for crystals, in preparation for the proof of Theorem~\ref{unique-min} on the fibers of the key map at longest elements of parabolic subgroups. 
Theorem~\ref{unique-min} will provide a vital  input into the proofs of our positive M\"obius function and homotopy type results, though it  can  be treated as a black box, e.g., by readers who are  more familiar with combinatorics  of posets than with  representation theory.  Section~\ref{sectkey} reviews the definition of the key map in terms of both the SSYT model and the alcove model, as well as abstract  properties of the key.   Section~\ref{sec:keyalg} 
 provides  a new, elementary algorithm for calculating the key, which is based on abstract properties of the key map.  Some readers may find it convenient  to take 
the output of  this algorithm as the  definition of the key map.  
Now let us commence with our   
 review of background material.  

\subsection*{Acknowledgments}

We thank Karola M\'esz\'aros, 
Jennifer Morse, 
Victor Reiner, Anne Schilling, Mark Shimozono, and John Stembridge for valuable discussions related to several aspects of this paper.

\section{Basic setup and background}

\subsection{Basic notions related to crystal graphs}\label{sec1}

Crystals are colored directed graphs encoding certain representations of quantum groups $U_q({\mathfrak g})$ in the limit $q\to 0$. We will only refer here to the highest weight integrable representations of symmetrizable Kac-Moody algebras, all of which are known to possess crystals. Given a dominant weight $\lambda\in\Lambda^+$ (where $\Lambda$ denotes the set of weights), we denote by $B=B(\lambda)$ the crystal of the irreducible representation $V(\lambda)$ of highest weight $\lambda$, which we call a {\em ${\mathfrak g}$-crystal}. Its vertices correspond to the so-called {\em crystal basis} of $V(\lambda)$, and its edges partially encode the action of the Chevalley generators $e_i,f_i$ of $U_q({\mathfrak g})$ on this basis; therefore, the edges are labeled by $i\in I$ (also called colors) indexing the corresponding simple roots $\alpha_i$. Note that if $\mathfrak g$ is a semisimple Lie algebra, then the corresponding root system is finite, and so are all the corresponding ${\mathfrak g}$-crystals $B(\lambda)$; otherwise, these crystals are infinite. 

We briefly recall the standard terminology related to root systems in Section~\ref{rootsys}, by focusing on the finite case, in preparation for presenting the alcove model for crystals (in Section~\ref{alcmod}). We now review some basic facts about crystals, and refer the reader to \cite{HK} for more details.  We start by defining a larger class of objects called crystals in an abstract way, and then we discuss which of them are ${\mathfrak g}$-crystals. 

\begin{definition}\label{defcr} {\rm A crystal $B$ is a directed graph with labeled edges satisfying axioms {\rm (P1)} and {\rm (P2)} below; the  vertex set is also denoted $B$, and the label on the edge $b\stackrel{i}{\rightarrow} b'$ is an element of $I$. 
\begin{enumerate}
\item[{\rm (P1)}] All monochromatic directed paths have finite length; in particular, $B$ has no monochromatic circuits.  
\item[{\rm (P2)}] For every vertex $b$ and every color $i\in I$, there is at most one edge $b\rightarrow b'$ with color $i$, and at most one edge $b''\rightarrow b$ with color $i$.
\end{enumerate}}
\end{definition}

The crystal operators $e_i,f_i : B\rightarrow B\cup 
\{ \Bzero \}$ for $i\in I$ (where $\Bzero \not \in B$) are defined by letting $f_i(b)= b'$ and $e_i(b')=b$ for each $i$-colored edge $b\rightarrow b'$; if there is no such edge, we set  $f_i(b)=e_i(b')=\Bzero$. Define the $i$-string through $b\in B$ to be the maximal path of the form $$e_i^{d}(b) \rightarrow\cdots
\rightarrow e_i(b) \rightarrow b \rightarrow f_i(b) \rightarrow\cdots\rightarrow f_i^r(b)$$ for some $r,d\ge 0$.   Then the {\em $i$-rise} of $b$, denoted $\varepsilon (b,i)$, is the nonnegative integer $r$, while the {\em $i$-depth} of $b$, denoted $\delta (b,i)$, is the nonpositive integer $-d$. 

The $\g$-crystals are also endowed with a weight function $\wt:B \to \Lambda$, which is subject to the condition: 
\[\wt(f_i(b))=\wt(b)-\alpha_i\,.\]
Using the weight function, we can express combinatorially the character of the representation $V(\lambda)$ as a positive linear combination of formal exponentials in the group algebra of the corresponding weight lattice:
\begin{equation}\label{charf}{\rm ch}\,V(\lambda)=\sum_{b\in B(\lambda)} x^{\wt(b)}\,.\end{equation}

 Now recall that, for a Weyl group element $w$, the {\em Demazure module} $V_w(\lambda)$ is the submodule of $V(\lambda)$ generated  from  an  extremal
weight  vector  of weight $w\lambda$ by  the  action  of  an  upper  triangular  subalgebra $U_q^+(\mathfrak{g})$ of the quantum group $U_q(\mathfrak{g})$. The corresponding {\em Demazure subcrystal} of $B(\lambda)$ is denoted $B_w(\lambda)$, and the Demazure character ${\rm ch}\,V_w(\lambda)$ is expressed in a similar way to \eqref{charf}, as a sum over $B_w(\lambda)$. Similarly, one defines the {\em opposite} Demazure module and crystal by using the lower triangular subalgebra $U_q^-(\mathfrak{g})$ of $U_q(\mathfrak{g})$.

Besides character formulas, crystals are used to derive combinatorial formulas for decomposing tensor products of irreducible representations, as well as their inductions and restrictions. 

We are interested in the partially ordered set structure on a $\g$-crystal, called a {\em crystal poset}, which is defined by its cover relations $b\lessdot b'$ with $b'=f_i(b)$ for some $i$. The crystal poset $B(\lambda)$ always has a minimum $\hat{0}$, satisfying $\wt(\hat{0})=\lambda$. On the other hand, the maximum $\hat{1}$ only exists in the finite case, in which case  we have $\wt(\hat{1})=w_\circ\lambda$, where $w_\circ$ is the longest element of the corresponding (finite) Weyl group. Note that this convention is opposite to the one in representation theory, where the minimum and maximum here are the so-called highest and lowest weight vertices of the $\g$-crystal, respectively. We will study crystal posets using poset theoretic  techniques. 
See Section ~\ref{poset-topology-bg} or see \cite{St} for further background on posets. 

\begin{remarks}\label{labint}{\rm (1) Based on the weight function, it is easy to see that the crystal posets are graded. In addition, all the maximal chains in a given crystal poset interval have the same multiset of edge labels. We will use these facts implicitly.

(2) Finite type $\g$-crystals are self-dual posets. More specifically, there is the so-called {\em Lusztig involution} $S$ on the $\g$-crystal, with the property: 
\[S(f_i(b))=e_{i^*}(S(b))\,,\]
where $i\mapsto i^*$ is the permutation of the simple roots determined by the longest Weyl group element: $w_\circ(\alpha_i)=-\alpha_{i^*}$. 

(3) One subtlety deserving particular attention  
 is that the crystal posets are not always 
 lattices; see Example~\ref{non-lattice}. \color{black}Note, however, that $A_2$-crystal posets are lattices, cf. \cite[Remark~4 and Proposition~5.4]{DKK0}.\color{black}

(4) Besides the highest weight crystals considered in this paper, there are other types of crystals. For instance, we have the {\em Kirillov-Reshetikhin (KR) crystals} for affine Lie algebras, which are finite, as opposed to the affine highest weight crystals. However, the KR crystals do not have a poset structure. 
}
\end{remarks}

\subsection{Characterization of $\g$-crystals}\label{local-section}
 We recall Stembridge's local characterization of $\g$-crystals in the case of simply laced root systems \cite{Stembridge}, \color{black} as well as the equivalent characterization in \cite{DKK0}; it is useful to note that the Stembridge relations are crystal graph analogues of the Verma relations. For doubly laced $\g$-crystals, we have a partial result in \cite{Sternberg}, and the corresponding local characterization in \cite{DKK}. \color{black} 

 We start with some notation. As usual, we denote by $a_{ij}$ the entries of the corresponding Cartan matrix, i.e., $a_{ij}=\langle\alpha_j,\alpha_i^\vee\rangle$. Recalling the notions of $i$-depth and $i$-rise introduced in Section~\ref{sec1}, we will use the symbol $\Delta_i$ as follows: 
\[\Delta_i \delta (b,j) := \delta (b',j) - \delta (b,j)\;\;\;\mbox{and}\;\;\;\Delta_i \varepsilon (b,j) := \varepsilon (b',j) - \varepsilon (b,j)\,,\;\;\;\mbox{whenever}\;\;\;b':=e_i (b)\ne\Bzero\,.\]
 Likewise, we set 
\[\nabla_i \delta (b,j) := \delta (b,j) - \delta (b',j)\;\;\;\mbox{and}\;\;\;\nabla_i \varepsilon (b,j) := \varepsilon (b,j) - \varepsilon (b',j)\,,\;\;\;\mbox{whenever}\;\;\;b':=f_i (b)\ne\Bzero\,.\]

In addition to axioms (P1) and (P2) in Definition~\ref{defcr}, Stembridge introduced the following axioms (P3)$-$(P6). For (P3) and (P4), we assume that $e_i(b)\ne\Bzero$, whereas for (P5) and (P6) we assume that $e_i(b)\ne\Bzero$ and $e_j(b)\ne\Bzero$.
\begin{enumerate}
\item[(P3)] $\Delta_i \delta(b,j) + \Delta_i \varepsilon (b,j) = a_{ij}$ for all $i,j\in I$.  
\item[(P4)] $\Delta_i \delta (b,j) \le 0$ and $\Delta_i \varepsilon (b,j)\le 0$ for all $i\ne j$.  
\item[(P5)]  $\Delta_i \delta (b,j) = 0$ implies $e_ie_j(b) = e_je_i(b)$ and $\nabla_j \varepsilon (b',i)=0$ where 
$b':=e_ie_j(b) = e_je_i(b)$.
\item[(P6)]  $\Delta_i \delta (b,j) = \Delta_j \delta (b,i) = -1$ implies $e_ie_j^2e_i (b) = e_j e_i^2 e_j(b)$ and $\nabla_i \varepsilon (b',j) = \nabla_j \varepsilon (b',i) = -1$ where $b':=e_ie_j^2e_i(b) = e_je_i^2e_j (b)$.
\end{enumerate}

Stembridge also formulated axioms (P$5'$) and (P$6'$) by replacing $e_i$ with $f_i$; these axioms are equivalent to (P5) and (P6).



\begin{theorem}{\rm \cite{Stembridge}} Axioms {\rm (P1)$-$(P6)} hold for any $\g$-crystal. These axioms characterize $\g$-crystals in the case of simply laced root systems. 
\end{theorem}

In simply laced types we have $a_{ij}\in\{0,-1\}$ for $i\ne j$, so (P3) and (P4) allow for only three possibilities:
\[(a_{ij},\,\Delta_i \delta(b,j),\,\Delta_i \varepsilon (b,j))\,=\,(0,0,0),\;(-1,-1,0),\;(-1,0,-1)\,.\]
Thus, the local structure is controlled by (P5) and (P6), cf.  the corollary below which also motivates Definition~\ref{stembridge-move-definition} later in the paper.

\begin{corollary}\label{corsl}
For any cover relations $u\stackrel{i}{\rightarrow}v$ and $u\stackrel{j}{\rightarrow} w$ with $i\ne  j$ in a simply laced $\g$-crystal, we have one of the following two cases:
\begin{enumerate}
\item[{\rm (i)}] there is $x$ which covers both $v$ and $w$ with $v\stackrel{j}{\rightarrow}x$ and $w\stackrel{i}{\rightarrow}x$;
\item[{\rm (ii)}] there are saturated chains $v \stackrel{j}{\rightarrow} v_1 \stackrel{j}{\rightarrow} v_2 \stackrel{i}{\rightarrow} x$ and $w\stackrel{i}{\rightarrow} w_1 \stackrel{i}{\rightarrow} w_2 \stackrel{j}{\rightarrow} x$, and no other such chains of length at most $3$.
\end{enumerate}
\end{corollary}

Stembridge proposed a larger set of axioms which conjecturally characterize doubly laced $\g$-crystals; namely, in addition to the expressions of length $2$ and $4$ in (P5) and (P6), there are similar expressions of length $5$ and $7$. Sternberg \cite{Sternberg} showed that these axioms are indeed satisfied by doubly laced $\g$-crystals, so we have the following analogue of Corollary~\ref{corsl}; see also \cite{DKK}. 

\begin{corollary}\label{cordl} For any cover relations $u\stackrel{i}{\rightarrow}v$ and $u\stackrel{j}{\rightarrow} w$ with $i\ne  j$ in a doubly laced $\g$-crystal, there is an upper bound $x$ of $v$ and $w$ in the label $\{i,j\}$-restricted subposet, with $d:={\rm rank}(x)-{\rm rank}(u)\in\{2,4,5,7\}$. In each case, we have a complete description of all the saturated chains from $v$ and $w$ to $x$, while no similar chains of length at most $d-1$ exist (when $d\in\{2,4\}$, we have precisely the cases in Corollary~{\rm \ref{corsl}}).
\end{corollary} 

\begin{remarks}{\rm (1) 
The results of Stembridge and Sternberg mentioned above, for simply and doubly laced $\g$-crystals, include all the corresponding finite, affine, and twisted affine cases. The only such $\g$-crystals not covered are the triply laced ones. Calculations show that these are much more complex, involving   over $40$ local moves. 

(2) As a word of caution, these results of Stembridge above do not guarantee that the upper (or lower) bound described in  his results will be the unique least upper bound or greatest lower bound.  Indeed, we construct an example of a crystal that is not a lattice in Example ~\ref{non-lattice}.   The existence of such examples is what makes possible the more extensive array of relations among crystal operators constructed e.g. in Section ~\ref{negative-section}, in spite of Stembridge's results.
}
\end{remarks}
	
	For simplicity, from now on we will refer to $\g$-crystals simply as crystals.

\subsection{The tableau model for crystals of type $A_{n-1}$}\label{ssyt} 

In this case, a dominant weight $\lambda$ can be viewed as a partition $(\lambda_1\ge\lambda_2\ge\ldots\ge\lambda_{n-1}\ge 0)$, or as the corresponding Young diagram. The vertices of $B(\lambda)$ can be indexed by all the {\em semistandard Young tableaux} (SSYT) of shape $\lambda$ with entries in $[n]:=\{1,\ldots,n\}$ (i.e., the entries in the boxes of $\lambda$ are weakly increasing in rows and strictly increasing in columns).

We will now describe the crystal operators on SSYT in terms of the so-called 
{\em signature rule}, which is just a translation of the well-known tensor product rule for arbitrary crystals \cite{HK}. 
To apply $f_i$ (or $e_i$) on $T$ in $B(\lambda)$, consider the word with letters
$i$ and $i+1$ formed by recording these letters in the columns of $T$, which are scanned from left to right and bottom to top. We replace the letter $i$ with the symbol $+$ and the letter $i+1$ with $-$. 
Then, we remove from our binary word adjacent pairs $-+$, as long as this is possible. At the end of this process, we are left with a word
\begin{equation}
    \rho_i(T) = \underbrace{++ \ldots +}_x\underbrace{-- \ldots -}_y\,,
    \label{eqn:reduced_word}
\end{equation}
called the $i$-signature of $T$.
\begin{definition} {\rm 
{\rm (1)} If $y > 0$, then $e_i(T)$ is obtained by replacing  in $T$ the letter 
        $i+1$ which corresponds to the leftmost $-$ in
        $\rho_i(T)$ with the letter $i$. 
        If $y=0$, then $e_i(T) = \Bzero$.

 {\rm (2)} If $x > 0$, then $f_i(T)$ is obtained by replacing  in $T$ the
        letter $i$ which corresponds to the rightmost $+$ in
        $\rho_i(T)$ with the letter $i+1$.
        If $x=0$, then $f_i(T) =\Bzero$. }
\label{definition:crystal_operators}
\end{definition}

\begin{example}\label{signature-example}
{\rm 
		Let $n=4$, 	 and $T = \tableau{1 & 2 & 2 & 2 & 2 & 3 \\ 3 & 3 & 4 }$, with $2$-signature $++-$. So we have $f_2(T)=\tableau{1 & 2 & 2 & 2 & 3 & 3 \\ 3 & 3 & 4 }\,$.}
\end{example}



\subsection{Finite root systems}\label{rootsys}
Let $\mathfrak{g}$ be a complex semisimple Lie algebra, and $\mathfrak{h}$ a Cartan subalgebra, whose rank is $r$. 
Let $\Phi \subset \mathfrak{h}^*$ be the corresponding irreducible \emph{root system}, 
$\mathfrak{h}^*_{\reals}\subset \mathfrak{h^*}$ the real span of the roots, and $\Phi^{+} \subset \Phi$ the set of positive roots.
Let $\Phi^{-} := \Phi \backslash \Phi^{+}$.
For $ \alpha \in \Phi$ we will say that $ \alpha > 0 $ if $ \alpha \in \Phi^{+}$,
and $ \alpha < 0 $ if $ \alpha \in \Phi^{-}$.
The sign of the root $\alpha$, denoted $\sgn(\alpha)$, is defined to be $1$ if $\alpha \in \Phi^{+}$, and $-1$ otherwise.
Let $| \alpha | = \sgn( \alpha ) \alpha $.
Let $\rho := \frac{1}{2}(\sum_{\alpha \in \Phi^{+}}\alpha)$. Let 
$\alpha_1,\ldots,\alpha_r\in \Phi^{+}$ be the corresponding \emph{simple roots}, and $s_i:=s_{\alpha_i}$ the corresponding simple reflections; the indexing set $\{1,\ldots,r\}$ for the simple roots is traditionally denoted $I$. We denote by 
$\inner{\cdot}{\cdot}$ the nondegenerate scalar product on 
$\mathfrak{h}^{*}_{\reals}$ induced by the Killing form. Given a root $\alpha$, we consider the corresponding \emph{coroot} $\alpha^{\vee} := 2\alpha/ \inner{\alpha}{\alpha}$ and reflection
$s_{\alpha}$.  

Let $W$ be the corresponding \emph{Weyl group}. The length function on $W$ is denoted by $\ell(\cdot)$. The longest element of $W$ is denoted by $w_\circ(I)$ or just $w_\circ$; 
we have $\ell(w_\circ)=|\Phi^+|$ (where $|\cdot|$ denotes cardinality), which is traditionally denoted $N$. A reduced expression $s_{i_1}\ldots s_{i_N}$ for $w_\circ$ determines a so-called {\em reflection order} $(\beta_1,\beta_2,\ldots,\beta_N)$ on the positive roots $\Phi^+$, where $\beta_j=s_{i_1}\ldots s_{i_{j-1}}(\alpha_{i_j})$. The \emph{Bruhat order} on $W$, sometimes called the strong Bruhat order, is defined by its covers $w \lessdot ws_{\alpha}$, for $\alpha \in \Phi^{+}$,
if $\ell(ws_{\alpha}) = \ell(w) + 1$; we write $w\stackrel{\alpha}{\longrightarrow}ws_\alpha$ for the corresponding edge of the Hasse diagram.  The strong Bruhat order may equivalently be defined via cover relations $w\lessdot s_\alpha w$ for $\ell (s_\alpha w) =  \ell (w) + 1$, yielding the same poset.

The \emph{weight lattice} $\Lambda$ is given by 
\begin{equation}
	\Lambda := \left\{ \lambda \in \mathfrak{h}_{\reals}^{*} \: : \:
	\inner{\lambda}{\alpha^{\vee}} \in \integers \text{ for any } \alpha \in \Phi \right\}\,.
	\label{eqn:weight_lattice}
\end{equation}
The weight lattice $\Lambda$ is generated by the \emph{fundamental weights} 
$\omega_1, \ldots \omega_r$, which form the dual basis to the basis of simple coroots, i.e., 
$\inner{\omega_i}{\alpha_j^{\vee}}= \delta_{ij}$. The set $\Lambda^{+}$ of \emph{dominant weights}
is given by 
\begin{equation}
	\Lambda^{+} := \left\{  \lambda \in \Lambda \: : \: 
	\inner{\lambda}{\alpha^{\vee}} \geq 0 \text{ for any } \alpha \in \Phi^{+}
	\right\}.
	\label{eqn:dominant_weights}
\end{equation}
A dominant weight is called {\em regular} if it is not on the walls of the dominant chamber, i.e., if $\langle\lambda,\alpha_i^\vee\rangle\ne 0$ for all simple roots $\alpha_i$; otherwise, it is called {\em non-regular}.

For $J\subseteq I=\{1,\ldots,r\}$, we let $W_J$, $\Phi_J$, and $w_\circ(J)$ be the corresponding parabolic subgroup of $W$, its root system, and its longest element, respectively.

Given $\alpha \in \Phi$ and $k \in \integers$, we denote by $s_{\alpha,k}$ the reflection in the affine hyperplane
\begin{equation}
	H_{\alpha,k}:= \left\{  \lambda \in \mathfrak{h}^{*}_{\reals} \: : \: 
	\inner{\lambda}{\alpha^{\vee}} = k \right\}
	\label{eqn:affine_hyperplane}.
\end{equation}
These reflections generate the \emph{affine Weyl group} $W_{\textrm{aff}}$ for the 
\emph{dual root system} \linebreak $\Phi^{\vee}:= \left\{ \alpha^{\vee} \: :\: \alpha \in \Phi \right\}$.
The hyperplanes $H_{\alpha,k}$ divide the real vector space $\mathfrak{h}^{*}_\reals$ into open regions, called \emph{alcoves.} The \emph{fundamental alcove} $A_{\circ}$ is given by
\begin{equation}
	A_{\circ} := \left\{ \lambda \in \mathfrak{h}_{\reals}^{*} \: : \: 
	0 < \inner{\lambda}{\alpha^{\vee}} < 1 \text{ for all } \alpha \in \Phi^{+}
	\right\}.
	\label{eqn:fundamental_alcove}
\end{equation}

\subsection{The alcove model for crystals}\label{alcmod} 

We now review the alcove model \cite{LP,LP1} for crystals. This model was formulated in \cite{LP1} in the level of generality of 
symmetrizable Kac-Moody algebras, but we restrict here to the case of finite root systems, i.e., to crystals corresponding to the semisimple Lie algebras. 

We say that two alcoves are {adjacent} if they are distinct and have a common wall. Given a pair of adjacent alcoves $A$ and $B$, we write $A \stackrel{\beta}{\longrightarrow} B$ if the
common wall is of the form $H_{\beta,k}$ and the root
$\beta \in \Phi$ points in the direction from $A$ to $B$.

\begin{definition}{\rm \cite{LP}} {\rm 
	An  \emph{alcove path} is a {sequence of alcoves} $(A_0, A_1, \ldots, A_m)$ such that
	$A_{j-1}$ and $A_j$ are adjacent, for $j=1,\ldots m.$ We say that an alcove path 
	is \emph{reduced} if it has minimal length among all alcove paths from $A_0$ to $A_m$.}
	\end{definition}
	
	Let $A_{\lambda}=A_{\circ}+\lambda$ be the translation of the fundamental alcove $A_{\circ}$ by the weight $\lambda$.
	
	\begin{definition}\label{deflc}{\rm \cite{LP}} {\rm 
		The sequence of roots $(\beta_1, \beta_2, \dots, \beta_m)$ is called a
		\emph{$\lambda$-chain} if 
		\[	
		A_0=A_{\circ} \stackrel{-\beta_1}{\longrightarrow} A_1
		\stackrel{-\beta_2}{\longrightarrow}\dots 
		\stackrel{-\beta_m}{\longrightarrow} A_m=A_{-\lambda}\]
is a reduced alcove path.}
\end{definition}

We now fix a dominant weight $\lambda$ and an alcove path $\Pi=(A_0, \dots , A_m)$ from 
$A_0 = A_{\circ}$ to $A_m = A_{-\lambda}$. Note that $\Pi$ is determined by the corresponding $\lambda$-chain $\Gamma:=(\beta_1, \dots, \beta_m)$, which consists of positive roots. 
We let $r_i:=s_{\beta_i}$, and let $\widehat{r_i}$ be the affine reflection in the hyperplane containing the common face of $A_{i-1}$ and $A_i$, for $i=1, \ldots, m$; in other words, 
$\widehat{r}_i:= s_{\beta_i,-l_i}$, where $l_i:=|\left\{ j<i \, : \, \beta_j = \beta_i \right\} |$. 

Let $F=\left\{ j_1 < j_2 < \cdots < j_s \right\}$  be a subset of $[m]$. The elements of $F$ are called \emph{folding positions} for the reason which we now explain. We ``fold'' $\Pi$ in the hyperplanes corresponding to these positions, namely $H_{\beta_{k},-l_k}$ with $k=j_s,j_{s-1},\ldots,j_1$ in this order, each time applying the corresponding affine reflection to the tail of the current sequence of alcoves, viewed as a ``folded path''; see \cite{LP1} for more details. Like $\Pi$,  the final folded path can be recorded by a sequence of roots, namely 
$\Gamma(F)=\left( \gamma_1,\gamma_2, \dots, \gamma_m \right)$, where 
	 \begin{equation}\label{defw}
	 \gamma_k:=r_{j_1}r_{j_2}\dots r_{j_p}(\beta_k)\,,
	 \end{equation}
	  with $j_p$ the 
	 largest folding position less than $k$. 
	 We define $\gamma_{\infty} := r_{j_1}r_{j_2}\dots r_{j_s}(\rho)$.
	 Upon folding, the hyperplane separating the alcoves $A_{k-1}$ and $A_k$ in $\Pi$ is mapped to 
\begin{equation}\label{deflev}
H_{|\gamma_k|,-\l{k}}=\widehat{r}_{j_1}\widehat{r}_{j_2}\dots \widehat{r}_{j_p}(H_{\beta_k,-l_k})\,,
\end{equation}
for some $\l{k}$, which is defined by this relation.

\begin{definition}\label{defadm} {\rm 
{\rm (1)}	A subset $F=\left\{ j_1 < j_2 < \cdots < j_s \right\} \subseteq [m]$ (possibly empty)
 is an \emph{admissible subset} if we have the following saturated chain in the strong Bruhat order on $W$:
\begin{equation}
	\label{eqn:admissible}
	1 \stackrel{\beta_{j_1}}{\longrightarrow} r_{j_1} \stackrel{\beta_{j_2}}{\longrightarrow} r_{j_1}r_{j_2} 
	\stackrel{\beta_{j_3}}{\longrightarrow} \cdots \stackrel{\beta_{j_s}}{\longrightarrow} r_{j_1}r_{j_2}\cdots r_{j_s}\,.
\end{equation}
We call $\Gamma(F)$ an \emph{admissible folding}.

{\rm (2)} Given $F$ as above, we let
\[\mu(F):=-\widehat{r}_{j_1}\widehat{r}_{j_2}\ldots \widehat{r}_{j_s}(-\lambda)\,,\;\;\;\;\;\kappa(F):=r_{j_1}r_{j_2}\ldots r_{j_s}\,,\]
and call them  the \emph{weight} and the {\em key} of $F$, respectively.}
\end{definition}

We let $\A(\Gamma)$ be the collection of 
admissible subsets; when $\Gamma$ is fixed, like for the rest of this section, we also use the notation $\A(\lambda)$.

We now define the crystal operators on $\A( \lambda )$. 
Given $F\subseteq [m]$ and
$\alpha\in \Phi^+$, we will use the following notation:
	\[ 
	 I_\alpha = I_{\alpha}(F):= \left\{ i \in [m] \, : \, \gamma_i = \pm \alpha \right\}\,, \qquad 
\widehat{I}_\alpha = \widehat{I}_{\alpha}(F):= I_{\alpha} \cup \{\infty\}\,, 
\]
and $l_{\alpha}^{\infty}:=\inner{\mu(F)}{\alpha^{\vee}}$.
The following graphical representation of the heights $l_i^F$ for $i\in{I}_\alpha$ and $l_{\alpha}^{\infty}$ 
is useful for defining the crystal operators.
Let 
\[\widehat{I}_{\alpha}= \left\{ i_1 < i_2 < \dots < i_n <i_{n+1}=\infty \right\}\,
	\text{ and  }
	\varepsilon_i := 
	\begin{cases}
		\,\,\,\, 1 &\text{ if } i \not \in F\\
		-1 & \text { if } i \in F
	\end{cases}.\,
	\]
Given a positive root $\alpha$, we define the continuous piecewise-linear function 
$g_{\alpha}:[0,n+\frac{1}{2}] \to \reals$ by
\begin{equation}
	\label{eqn:piecewise-linear_graph}
	g_\alpha(0)= -\frac{1}{2}, \;\;\; g'_{\alpha}(x)=
	\begin{cases}
		\sgn(\gamma_{i_k}) & \text{ if } x \in (k-1,k-\frac{1}{2}),\, k = 1, \ldots, n\\
		\varepsilon_{i_k}\sgn(\gamma_{i_k}) & 
		\text{ if } x \in (k-\frac{1}{2},k),\, k=1,\ldots,n \\
		\sgn(\inner{\gamma_{\infty}}{\alpha^{\vee}}) &
		\text{ if } x \in (n,n+\frac{1}{2}).
	\end{cases}
\end{equation}
It was proved in \cite{LP1} that
\begin{equation}
	\label{eqn:graph_height}
	\l{i_k}=g_\alpha\left(k-\frac{1}{2}\right), k=1, \dots, n, \, 
	\text{ and }\, 
	l_{\alpha}^{\infty}:=
	\inner{\mu(F)}{\alpha^{\vee}} = g_{\alpha}\left(n+\frac{1}{2}\right).
\end{equation}

We will need the following properties of admissible subsets, relative to a simple root $\alpha$, which will be used implicitly. Note first that the function $g_{\alpha}$ is determined by the sequence $(\sigma_1, \dots, \sigma_{n+1})$, where 
	$\sigma_j = (\sigma_{j,1},\sigma_{j,2}):= (\sgn(\gamma_{i_j}), \varepsilon_{i_j}\sgn (\gamma_{i_j}))$ for 
	$1\leq j\leq n$,  and $\sigma_{n+1}= \sigma_{n+1,1}:=\sgn (\langle \gamma_{\infty}, \alpha^{\vee} \rangle)$. It was proved in \cite{LP1} that we have the following restrictions:
	\begin{enumerate}
	\item[(C1)]   $\sigma_{1,1}=1$.  
	\item[(C2)] $\sigma_{j}\in\{(1,1),\,(-1,-1),\,(1,-1)\}$ for $j\le n$.
	\item[(C3)]	$\sigma_{j,2}=1 \Rightarrow \sigma_{j+1,1}=1$.
	\end{enumerate}
These restrictions imply that if $g_\alpha$ attains its maximum $M$ at $x$, then $M \in \ZZ_{\geq 0}$, $x=m+\frac{1}{2}$ for $0 \leq m \leq n$, 
	  and $\sigma_{m+1} \in \left\{ (1,-1),1 \right\}$.

We will now define the crystal operators on an admissible subset $F$. 
Fix $p$ in $\{1,\ldots,r\}$, so $\alpha_p$ is a simple root.
Let $M=M(F,p)\ge 0$ be the maximum of $g_{{\alpha}_p}$ corresponding to $F$. Assuming that $M>0$, let $m$ be the minimum index 
$i$  in $\widehat{I}_{{\alpha}_p}$ for which we have $\l{i}=M$ (cf. the remarks above). 
We have either $m\in F$ or $m=\infty$; furthermore, $m$ has a predecessor $k$ in $\widehat{I}_{{\alpha}_p}$, and it turns out that $k\not\in F$. We define
\begin{equation}
	\label{eqn:rootF} 
f_p(F):= 
	\begin{cases}
		(F \backslash \left\{ m \right\}) \cup \{ k \} & \text{ if $M>0 $ } \\
				\Bzero & \text{ otherwise }.
	\end{cases}
\end{equation}
Now we define $e_p$. Again let $M$ be the maximum of $g_{{\alpha}_p}$. Assuming that $M>\inner{\mu(F)}{{\alpha}_p^{\vee}}$, let $k$ be the maximum index 
$i$  in $I_{{\alpha}_p}$ for which we have $\l{i}=M$, which always exists. Let $m$ be the successor of $k$ in $\widehat{I}_{{\alpha}_p}$. It turns out that $k\in F$, and either $m\not\in F$ or $m=\infty$. 
Define 
\begin{equation}
	\label{eqn:rootE}
e_p(F):= 
	\begin{cases}
		(F \backslash \left\{ k \right\}) \cup \{ m \} & \text{ if }
		M>\inner{\mu(F)}{{\alpha}_p^{\vee}}  \\
				\Bzero & \text{ otherwise. }
	\end{cases}
\end{equation}
In the above definitions, we use the
convention that $F\backslash \left\{ \infty \right\}= F \cup \left\{ \infty \right\} = F$. 

It was shown in \cite{LP1} that, whenever $f_p(F)$ and $e_p(F)$ are not equal to $\Bzero$, they are also admissible subsets. The main result of \cite{LP1}, cf. also \cite{lenccg}, is that the combinatorially defined crystal $\A(\lambda)$ is isomorphic to the crystal corresponding to the highest weight representation $V(\lambda)$; we stress that this is true for {\em any} $\lambda$-chain. We will also need the following expressions of the rise and the depth (see Section~\ref{sec1}) in terms of the alcove model, based on the above notation:
\begin{equation}\label{risedepth}
\varepsilon(F,p)=M(F,p)\,,\;\;\;\;\delta(F,p)=\inner{\mu(F)}{{\alpha}_p^{\vee}}-M(F,p)\,.
\end{equation}

\subsection{The key of a crystal}\label{sectkey}
Let us now discuss the (right) key,  viewed as a map from the crystal to the Weyl group $W$.  More specifically, it is a poset map to the left weak Bruhat order. Recall that the {\em left weak Bruhat 
order}  is defined by its cover
relations $w\lessdot s_i w$ which holds whenever $\ell(s_iw)>\ell(w)$. In what follows, the context will always specify whether we refer to the strong or the left weak Bruhat order.

The (right) key was first defined in type $A$ by Lascoux and Sch\"utzenberger \cite{LS} as a special SSYT (with the entries in each column contained in the previous one) associated to an arbitrary SSYT of the same shape $\lambda$. This can be viewed in a natural way as a permutation in $S_n^\lambda$, that is, the subset of $S_n$ consisting of lowest coset representatives modulo the stabilizer of $\lambda$. The construction is in terms of the corresponding {\em plactic monoid}. See also \cite{Reiner-Shimozono}.

The key was later defined in arbitrary symmetrizable Kac-Moody type in terms of the model for crystals based on {\em Lakshmibai-Seshadri (LS) paths} \cite{Li}. Here it is known as the {\em initial direction} of a path. In the alcove model (in the finite case), the key appears in Definition~\ref{defadm}~(2); we will use the notation there, namely $\kappa(b)$ is the key of a crystal vertex $b$. It is not hard to see that the chain \eqref{eqn:admissible} consists of lowest coset representatives modulo the stabilizer $W_\lambda$ of $\lambda$, so the key map takes values in this set, denoted $W^\lambda$.

The importance of the key lies in the fact that it describes the Demazure subcrystal $B_w(\lambda)$ as consisting of those vertices $b$ in $B(\lambda)$ with $\kappa(b)\le w$ in strong Bruhat order. In other words, the key of $b$ in $B(\lambda)$ can be interpreted as giving the smallest Demazure subcrystal $B_w(\lambda)$ of $B(\lambda)$ containing $b$. We immediately get a combinatorial formula for Demazure characters generalizing \eqref{charf}. 


 The following two properties of the key map, which hold in symmetrizable Kac-Moody generality, are well-known, see \cite[Lemma~5.3]{Li} and the discussion below. First, if $f_p(F)\ne\Bzero$, then
\begin{equation}\label{changekey}\kappa(f_p(F))=\casetwo{\kappa(F)}{e_p(F)\ne\Bzero}{s_p\kappa(F)\,\;\mbox{or}\,\;\kappa(F)}{e_p(F)=\Bzero}\end{equation}
Secondly, if $e_p(F)=\Bzero$, then 
\begin{equation}\label{firstkey}s_p \kappa(F)>\kappa(F)\,.\end{equation}

\begin{remark}\label{key-is-poset-map}{\rm 
Properties \eqref{changekey} and \eqref{firstkey} immediately imply that the key map is a projection from the crystal poset to the left  weak Bruhat order on $W$, and in particular is a poset map $-$ a property we will make extensive use of later.}
\end{remark}

 It is not immediately clear that, without a model for the crystal, one can distinguish between the two cases in \eqref{changekey} corresponding to a vertex at the beginning of a $p$-string, i.e., that a recursive construction of the key map based only on the crystal structure exists. In Section~\ref{sec:keyalg},  we will develop and justify such a construction that does not refer to the choice of a model, by expressing the key exclusively in terms of the crystal structure (i.e. the poset equipped with an edge coloring for its cover relations).
 We should note that there are easy criteria in terms of various models for deciding which case we are in, when we wish to apply \eqref{changekey}; for instance, in the alcove model, $f_p$ changes the key if and only if $m=\infty$ (expressed in terms of  the notation above). 

Some further  references to the literature on the key map are needed at this point. In \cite{LP1}, the LS path model (as a crystal isomorphism) is bijected to the alcove model corresponding to a particular $\lambda$-chain, called the {\em lex $\lambda$-chain}; this bijection maps the initial direction of a path to the key. In \cite{lenccg} the second author showed that the alcove model is independent of the choice of $\lambda$-chain, and exhibited crystal isomorphisms between any two (combinatorial) crystals $\A(\Gamma)$; the second author  also showed that these isomorphisms are compatible with the key. In type $A$, in \cite{LL},  an explicit crystal isomorphism is constructed, called the {\em filling map}, between the alcove model for a particular $\lambda$-chain (not the lex one) and the model based on SSYT. The filling map leads to an easy construction of the key of a SSYT, see \cite{LL}. Willis \cite{Wi} showed that the filling map is compatible with the keys (where on the SSYT side we use the right key of Lascoux-Sch\"utzenberger, see Section~\ref{ssyt}). In fact, an efficient construction of the right key of a SSYT, called the {\em scanning method}, was given earlier by Willis (see \cite{Wi}), and shown in \cite{Wi} to be closely related to the method mentioned above, related to the filling map.  

Now we turn to the dual notion of left key.  Lascoux and Sch\"utzenberger also defined the {\em left key} of a SSYT. A vast generalization of it, in the symmetrizable Kac-Moody setup, is the {\em final direction} of an LS path \cite{Li}. In a similar way to the initial direction, the final one detects membership in an opposite Demazure crystal (see Section~\ref{sec1}). We now recall from \cite{lenccg} the construction of the left key for the alcove model (in the finite case), that is, the image of the final direction of an LS path under the bijection between LS paths and the alcove model mentioned in the previous paragraph. We always refer to the left key by the full name (to distinguish it from the key $\kappa(\,\cdot\,)$), and we denote it by $\overline{\kappa}(\,\cdot\,)$.  

We now use the finite case setup in Section \ref{alcmod}. We fix a dominant weight $\lambda$, an index set $I:=\{\overline{1}<\ldots<\overline{q}<1<\ldots<m\}$, and a corresponding $\lambda$-chain $\Gamma:=(\beta_{\overline{1}},\ldots,\beta_{\overline{q}},\beta_1,\ldots,\beta_m)$ such that $l_i=0$ if and only if $i\in \overline{I}:=\{\overline{1}<\ldots<\overline{q}\}$. In other words, the second occurrence of a root can never be before the first occurence of another root. We recall the notation $r_i:=s_{\beta_i}$ for $i\in I$. 

\begin{definition}\label{defkeys} {\rm Let $F$ be an admissible subset (with respect to the special $\lambda$-chain $\Gamma$ above). Let $F\cap \overline{I}=\{\overline{\jmath}_1<\ldots<\overline{\jmath}_a\}$. The {left key} $\overline{\kappa}(F)$ of $F$ is the Weyl group element defined by
\[\overline{\kappa}(F):=r_{\overline{\jmath}_1}\ldots r_{\overline{\jmath}_a}\,.\]}
\end{definition}

We clearly have $\overline{\kappa}(F)\le\kappa(F)$ (in strong Bruhat order) and, like $\kappa(F)$, the left key $\overline{\kappa}(F)$ also belongs to $W^\lambda$. By \cite[Corollary~6.2]{lenccg}, the Lusztig involution (see Remark~\ref{labint}~(2)) relates the left and the right key as follows; here $\floor{w}$ denotes the lowest representative of the coset $wW_\lambda$, and $w_\circ(\lambda)$ the longest element of $W_\lambda$.

\begin{theorem} {\rm \cite{lenccg}}\label{relkeys}
For any admissible subset $F$, we have
\begin{equation*} \overline{\kappa}(S(F))=\floor{w_\circ \kappa(F)}=w_\circ\kappa(F) w_\circ(\lambda)\,,\;\;\;\;\kappa(S(F))=\floor{w_\circ \overline{\kappa}(F)}=w_\circ \overline{\kappa}(F)w_\circ(\lambda)\,.\end{equation*}
\end{theorem}

  \subsection{Posets and poset topology}\label{poset-topology-bg}
  
  A partially ordered set (poset) is generated by {\it cover relations} or {\it covers}, denoted $u\lessdot v$ or $u\prec v$, namely ordered pairs $(u,v)$  satisfying $u\le v$ where $u\le z \le v$ implies $z=u$ or $z=v$.    A {\it saturated chain} from $u$ to $v$ is a series of cover relations $u = u_0 \lessdot u_1 \lessdot \cdots \lessdot u_k = v$.     A finite poset is {\it graded} if for each $u\le v$, all saturated chains from $u$ to $v$ have the same number of cover relations, with this number called the {\it rank} of the interval.  An {\it open interval}, denoted $(u,v)$, in a poset $P$ is the subposet comprised of elements $z$ satisfying $u<z<v$, while the {\it closed interval} $[u,v]$ is the subposet of elements $z$ satisfying $u\le z \le v$.   When a poset has a unique minimal element, we denote this by $\hat{0}$, while the unique maximal element, when it exists, is denoted $\hat{1}$.
  
 Recall that the {\it M\"obius function}  $\mu_P(x,y)$ of a partially ordered set (poset)  $P$ is defined recursively as follows: $\mu_P(x,x)=1$ for $x\in P$ and $\mu_P(x,y) = -\sum_{x\le z < y} \mu_P(x,z)$ for each $x<y$.  Sometimes we simply write $\mu (x,y)$ suppressing the $P$ when it is clear from context which poset is intended.    The {\it order complex} $\Delta (P)$ of a poset $P$ is the abstract simplicial complex whose $i$-dimensional faces are the chains $x_0 < x_1 < \cdots < x_i$ of comparable elements in $P$.  Denote by $|\Delta (P)|$ a  geometric realization of $ \Delta (P)$.  
 Let $\Delta_P (x,y)$, again with the $P$ sometimes suppressed, denote the order complex of the subposet $(x,y)$ comprising the open interval from $x$ to $y$.
 
  One reason for interest in poset order complexes is the interpretation of the M\"obius function given by $\mu_P(x,y) = \tilde{\chi }(\Delta_P(x,y))$ (cf. \cite{Rota}).  Thus, proving that $\Delta_P(x,y)$ is homotopy equivalent to a ball or a sphere will imply that $\mu_P(x,y)$ must equal either $0$ (in the case of a ball) or $\pm 1$ in the case of a sphere. The {\it face poset} $F(\Delta )$ of a simplicial complex $\Delta $ is the partial order on the faces of $\Delta $ by inclusion of the associated sets of vertices comprising the faces.     See e.g. \cite{St} for further background regarding posets and their M\"obius functions.
 
 
 A {\it poset map} is a map $f:P\rightarrow Q$ from a poset $P$ to a poset $Q$ 
  such that $x\le y$ in $P$ implies $f(x) \le f(y)$ in $Q$.   In this paper, the main poset map of interest will be the key, regarded as a poset map from a crystal to the weak Bruhat order.
  {The next result, due to Quillen, 
  appears in \cite{Quillen}.
 \begin{theorem}[Quillen Fiber Lemma]\label{Quillen-fiber-lemma}
 For $f:P\rightarrow Q$ a poset map such that 
 $f_{\le x}^{-1} = \{ z\in P \,:\, f(z) \le x \} $ is contractible for each $x\in Q$, then $\Delta (P) $ is homotopy equivalent to $\Delta (Q)$. 
  \end{theorem}
  A sufficient condition for contractibility is that $f^{-1}(x)$ has a unique maximal element for each $x\in Q$.  One point requiring case is that we are typically studying the topology of an open interval in a  poset, which means there is not a maximal  element or a minimal element in one or both of the posets $P$ and $Q$ between which we construct a map.    }
 
{
 \begin{remark}{\rm 
 The Quillen fiber lemma also has a dualized version with $f_{\le x}^{-1}$ replaced by $f_{\ge x}^{-1}$ throughout, reflecting the fact that the order complex of a poset is the same abstract simplicial complex as the order complex of the dual poset.  }
 \end{remark}
 }
 
 \section{A new recursive algorithm to calculate the key}\label{sec:keyalg}


In this section, we develop a recursive approach to calculating the (right) key map directly from the crystal graph, which works in the level of generality of  symmetrizable Kac-Moody algebras. 
%
%
%
This  algorithm proceeds from lower to higher ranks in the poset, assuming the key is known for all lower ranks and then showing how to deduce it for the next rank.  
This  algorithm  and the proof of its validity will rely on  the following properties of the key map, which we treat here as axioms:  

\begin{enumerate}
\item
$ \kappa(\hat{0}) = 1$;
\item
for each $u$ that covers $\hat{0}$, we have $\kappa (u) = s_i$ where $i$ is the color on the cover relation from $\hat{0}$ to $u$;
\item property \eqref{firstkey} from Section~\ref{sectkey}, {namely that $e_j(u)=\Bzero$ implies $s_j\kappa (u)> \kappa (u)$;
\item property \eqref{changekey} from Section~\ref{sectkey}, {which together with (3) implies that 
$\kappa$ is a poset map.}
}
\end{enumerate}

Axiom (2) is immediate, either based on a combinatorial model (like the alcove model) or the representation-theoretic interpretation of the key in terms of Demazure subcrystals (see Section~\ref{sectkey}).

 Other facts that follow are listed below. 

\begin{lemma}
If both $u$ and $u'$ are covered by $v$ with $\kappa (u) \ne \kappa (u')$, then $\kappa (v) = \kappa (u)\vee \kappa (u')$ where this join is taken  in the weak order.  In particular, if the length of $\kappa (u)$ is less than the length of $\kappa (u')$ then $\kappa (v)=\kappa (u')$, and on the other hand if $\kappa (u)$ has the same length as $\kappa (u')$ then $\kappa (v)$ must have length one more than each of these.
\end{lemma}

\begin{proof}
  This follows from $\kappa $ being a poset map. 
\end{proof}

\begin{lemma}
If both $u$ and $u'$ are covered by $v$ with $\kappa (u) = \kappa (u')$ then $\kappa (v) = \kappa (u) = \kappa (u')$.
\end{lemma}
\begin{proof} 
If $\kappa (v) > \kappa (u) = \kappa (u')$ then $\kappa (v) = s_i \kappa (u) = s_j \kappa (u') = s_j \kappa (u)$ where the cover relation from $u$ to $v$ is colored $i$ and the cover relation from $u'$ to $v$ is colored $j$, but this is a contradiction since $s_i \kappa (u)$ does not equal $s_j \kappa (u)$ for $i\ne j$.
\end{proof}

\begin{lemma}\label{key-for-unique-cover}
Suppose $u$ is covered by $v$ and there are no other elements $u'$ also covered by $v$.  Let 
$j$ be the color of the cover relation from $u$ to $v$.  
  Then $\kappa(v)=\kappa(u)$ if $e_j(u)\ne\Bzero$, while $\kappa(v)=s_j\kappa(u)$ otherwise.   
\end{lemma}

\begin{proof}
 Since $u$ is the only element covered by $v$, we have $s_i\kappa(v)>k(v)$ for all $i\ne j$, by \eqref{firstkey}. 
If $e_j(u)\ne \Bzero$, then $\kappa(v)=\kappa(u)$ by \eqref{changekey}, so consider the case when $e_j(u)=\Bzero$. Assume for contradiction 
that $\kappa(v)=\kappa(u)$ (recall that we want to prove the opposite in this case, namely $\kappa(v)=s_j\kappa(u)$). Using \eqref{firstkey} again and the above assumptions, specifically using the assumption that $\kappa(v) = \kappa (u)$, we deduce $s_j \kappa(v)>\kappa(v)$. We conclude that $\kappa(v)=1$, so $v=\hat{0}$ (by the above axioms), which is a contradiction.
\end{proof}

Putting together these lemmas yields the desired algorithm:

\begin{theorem}\label{first-key-algorithm}
The key is fully determined by the crystal via the following rules:
\begin{enumerate}
\item[{\rm (1)}]
$\kappa(\hat{0})=1$;
\item[{\rm (2)}]
$\kappa(a)=s_i$ for each atom $a$ with cover relation $\hat{0}\lessdot a$ colored $i$;
\item[{\rm (3)}]
$\kappa(v) = \vee_{\{ u\,|\, u\lessdot v \} } \kappa (u)$ with join taken in the left  weak Bruhat order, provided that $v$ covers two or more elements;
\item[{\rm (4)}]
$\kappa (v) = \kappa (u) $ if $u$ is the unique element of the crystal covered by $v$, provided that the cover relation $u\lessdot v$ is colored $i$ and there is also a cover relation $u'\lessdot u$ colored $i$;
\item[{\rm (5)}]
$\kappa (v) = s_i \kappa (u)$ if $u$ is the unique element of the crystal covered by $v$, provided that the cover relation $u\lessdot v$ is colored $i$ while there are no cover relations $u'\lessdot u$ colored $i$.
\end{enumerate}
 \end{theorem}

{
\begin{corollary}\label{keyiso}
 Axioms {\rm (1)$-$(4)} determine the key. Isomorphic crystal posets have  identical keys; they are identified by the isomorphism (which is unique).
\end{corollary}
}
\begin{remarks}{\rm 
(1) Typically {in the literature} the key has been described in terms of  data associated to a choice of model for the crystal, for instance the SSYT model in type $A$ or the alcove model more generally.  We have just shown that it can be calculated without reference to  such data coming from a model for the crystal.  

(2) Corollary~\ref{keyiso} has the following application. We often construct crystal isomorphisms between two different models for crystals; an example is the filling map mentioned in Section~\ref{sectkey}, which maps the type $A$ alcove model to the SSYT model, see~\cite[Theorem~4.7]{LL}. However, the key constructions in the two models are often quite different, like in the mentioned example; in this case, it could be hard to show, based on these constructions, that the crystal isomorphism identifies the keys. This fact is now immediate. 

(3) One may construct a similar algorithm to calculate the left key in a finite type crystal by proceeding instead  from top to bottom through the crystal. This is based on similar axioms to (1)$-$(4). Indeed, $\overline{\kappa}(\hat{1})$ is the lowest  representative of the coset $w_\circ W_\lambda$ (see, e.g., \cite[Proposition~5.1]{lenccg}), and the analogues of \eqref{changekey} and \eqref{firstkey} are also standard (see, e.g., \cite[Proposition~3.19]{LeSh}). 
}
\end{remarks}

\section{Connectedness of saturated chains for lower and upper intervals in a  crystal poset}
 
 
 In this section, we will 
 %
 prove a crystal theoretic analogue of the \color{black} word property of Coxeter groups (any two reduced expressions for the same  Coxeter group 
 element are connected by a series of braid moves). \color{black} 
Results of Stembridge, and in  particular  Corollary~\ref{corsl},  lead us to establish Definition 
~\ref{stembridge-move-definition} below to serve as a suitable analogue for braid moves for crystals  in the simply laced case.  {Results of Sternberg will allow this notion and the subsequent Theorem~\ref{connect-moves} to be extended to the doubly laced case using exactly  the same line of argument (cf. Remark ~\ref{connected-remarks}, part (2)).}
 
 \begin{definition}\label{stembridge-move-definition} {\rm 
 Let us define a {\rm Stembridge move} on a simply laced
 crystal  as either the  replacement of a  saturated chain segment 
 $x\lessdot f_{c_2} (x) \lessdot f_{c_1}(f_{c_2}(x))$ by the saturated chain segment  
 $x \lessdot  f_{c_1}(x) \lessdot    f_{c_2}(f_{c_1}(x))$ in the event that 
 $f_{c_1}(f_{c_2}(x)) = f_{c_2}(f_{c_1}(x))$  or  the replacement of a saturated chain segment 
 $$x \lessdot f_{c_2}(x) \lessdot f_{c_1}(f_{c_2}(x)) \lessdot f_{c_1}(f_{c_1}(f_{c_2}(x))) \lessdot f_{c_2}(f_{c_1}(f_{c_1}(f_{c_2}(x))))$$ with the saturated chain segment 
 $$x \lessdot f_{c_1}(x) \lessdot f_{c_2}(f_{c_1}(x)) \lessdot f_{c_2}(f_{c_2}(f_{c_1}(x))) \lessdot f_{c_1}(f_{c_2}(f_{c_2}(f_{c_1}(x))))$$  in the event that $f_{c_2}(f_{c_1}(f_{c_1}(f_{c_2}(x)))) = f_{c_1}(f_{c_2}(f_{c_2}(f_{c_1}(x))))$.}
 \end{definition}
 
 
  \begin{theorem}\label{connect-moves}
 In the simply laced case, any two saturated chains in a lower interval $[\hat{0},v]$ are connected by a series of Stembridge moves. In addition, in finite type, the same result holds for upper intervals $[v,\hat{1}]$. 
 \end{theorem}
 
 \begin{proof}
 Use induction on the rank of $v$, which is well-defined by virtue of Remark~\ref{labint}~(1).  Now consider $m_1,m_2$ maximal chains with $x\lessdot v$ in $m_1$ colored $c_1$ and $y\lessdot v$ in $m_2$ colored $c_2$.  Denote by $x\wedge y$ the unique element that either (a) is covered by $x$ and $y$ with $f_{c_1}(x\wedge y) = y$ and $f_{c_2}(x\wedge y) = x$ or (b) is less than both $x$ and $y$ with $f_{c_2}(f_{c_2}(f_{c_1}(x\wedge y))) = x$ and $f_{c_1}(f_{c_1}(f_{c_2}(x\wedge y))) = y$.  Now $\hat{0} \le x\wedge y$, so let $n$ be any saturated chain from $\hat{0} $ to $x\wedge y$.  Let $m_3 $ (resp. $m_4$)  be a saturated chain from $\hat{0} $ to $v$ which includes $n$ and includes $x\lessdot v$ (resp. $y\lessdot v$).  By induction on the rank of $v$, we are given that $m_1$ (resp. $m_2$) is connected by Stembridge moves to $m_3$ (resp. $m_4$).  By construction, $m_3$ is connected to $m_4$ by a single Stembridge move.  Thus, $m_1$ is connected to $m_3$ is connected to $m_4$ is connected to $m_2$ by Stembridge moves, implying that $m_1$ is connected to $m_2$ by Stembridge moves, completing the proof for lower intervals. 

For upper intervals in finite type, a completely similar argument works. Alternatively, we reduce to lower intervals via Lusztig's involution on the corresponding crystal, see Remark~\ref{labint}~(2).
 \end{proof}
 
\begin{remarks}\label{connected-remarks}
{\rm 
 (1) The proof of Theorem~{\rm \ref{connect-moves}}
  is similar to the proofs of Theorem~{\rm 3.3.1} in  \cite{BB} and Lemma~{\rm 4.4} in \cite{HM}, dealing with the weak order and with more general posets admitting so-called $SB$-labelings, respectively.

 (2) Whenever there are local moves induced by $u\lessdot v$ and $u\lessdot w$, this proof above will hold.   In particular, this will allow a similar result in the doubly laced case via results of Sternberg providing the requisite moves in that setting.

(3) The proof of Theorem~{\rm \ref{connect-moves}} can easily be turned into a recursive algorithm for determining the sequence of moves connecting any pair of  saturated chains by proceeding from top to bottom so as to determine the requisite Stembridge moves.
}
\end{remarks}

\section{The fibers of the key map}\label{key-fibers-section}

We work in the generality of a root system $\Phi$ of rank $r$ for a symmetrizable Kac-Moody algebra $\mathfrak{g}$, and consider a highest weight $\mathfrak{g}$-crystal $B(\lambda)$. Let $W_K$, for $K\subseteq I=\{1,\ldots,r\}$, be the stabilizer of $\lambda$ (as a parabolic subgroup of $W$), which was previously denoted $W_\lambda$; in particular, $K=\emptyset$ if and only if $\lambda$ is regular. Recall the notation $\floor{w}$ for the lowest representative of the coset $wW_\lambda=wW_K$, and let $K^c$ denote the complement of $K$ in $I$. Fix a subset $J$ of $I$ such that $\Phi_J$ is a finite root system, and recall the notation $w_\circ(J)$ for the longest element in the parabolic subgroup $W_J$ of $W$. We denote by $\mathfrak{g}_J$ the Lie algebra corresponding to $\Phi_J$, and by $\lambda_J$ the projection of $\lambda$ to the weight lattice corresponding to $\Phi_J$. If $\Phi$ is a finite root system, let $w_\circ^J$ be the lowest representative $w_\circ(I) w_\circ(J)$ of the coset $w_\circ W_{J}$. 

 \begin{theorem}\label{unique-min}
   The fiber of the key map at $w_\circ(J)$ is non-empty if and only if $J\subseteq K^c$. Assuming this, the fiber has a minimum and a maximum, relative to the induced order from the crystal poset.  
 \end{theorem}
 
We will need the following two lemmas. For the first one, we consider a reflection order $(\beta_1,\ldots,\beta_N)$ on the positive roots of a finite root system $\Phi$, and let $r_i:=s_{\beta_i}$. 

\begin{lemma}\label{lemreford}
The root
\[r_1 r_2\ldots r_{j-1}(\beta_j)=-w_\circ r_N r_{N-1}\ldots r_{j+1}(\beta_j)\]
is a simple root.
\end{lemma} 

\begin{proof} Recall that a reflection order is determined by a reduced word $s_{i_1}\ldots s_{i_N}$ for $w_\circ$, in the sense that $\beta_j=s_{i_1}\ldots s_{i_{j-1}}(\alpha_{i_j})$. It follows that $r_j=s_{i_1}\ldots s_{i_{j-1}}s_{i_j}s_{i_{j-1}}\ldots s_{i_1}$ and 
\[r_1 \ldots r_{j-1}(\beta_j)=s_{i_{j-1}}\ldots s_{i_1}\,(s_{i_1}\ldots s_{i_{j-1}}(\alpha_{i_j}))=\alpha_{i_j}\,.\]
Furthermore, since $r_1\ldots r_N=w_\circ$, we also have $w_\circ r_N\ldots r_{j+1}=r_1\ldots r_j$, which leads to the conclusion of the proof.
\end{proof}

For the second lemma, we need the generalization of the alcove model to the symmetrizable Kac-Moody case in \cite{LP1}. This looks formally the same as the model in the finite case, described in Section~\ref{alcmod}. The main difference is that the $\lambda$-chain is infinite, and is no longer defined in terms of alcoves (but in terms of an equivalent condition, which generalizes to the infinite case). 

\begin{lemma}\label{replem}  The Demazure $\mathfrak{g}$-crystal $B_{w_\circ(J)}(\lambda)$ is isomorphic to the highest weight $\mathfrak{g}_J$-crystal $B(\lambda_J)$.
\end{lemma}

\begin{proof} Consider the restriction of $B(\lambda)$ to $U_q(\mathfrak{g}_J)$, that is, to edge labels in $J$. Let $B(\lambda,J)$ be the component of the highest weight vertex of $B(\lambda)$ in the restriction. By comparing highest weights, it is clear that $B(\lambda,J)\simeq B(\lambda_J)$ as $\mathfrak{g}_J$-crystals.

Let us now compare the $\mathfrak{g}_J$-crystal $B(\lambda,J)$ to the Demazure $\mathfrak{g}$-crystal $B_{w_\circ(J)}(\lambda)$. We first check that their  vertex sets coincide as subsets of $B(\lambda)$, by using the generalized alcove model to describe the latter (see above). The elements $F\in B(\lambda)$ which belong to $B_{w_\circ(J)}(\lambda)$ are described by the condition $\kappa(F)\le w_\circ(J)$ in strong Bruhat order (see Section~\ref{sectkey}). This condition implies (and is in fact equivalent to the fact) that the roots of the $\lambda$-chain corresponding to $F$ are all in $\Phi_J$. Indeed, all elements $w$ in the chain \eqref{eqn:admissible} are below $w_\circ(J)$, which implies that they are in $W_J$ (by the subword property of the strong Bruhat order). 

On the other hand, the elements $F$ of the crystal $B(\lambda,J)$ are also described by the condition mentioned above. To prove this, first note that any such  $F$ is obtained by acting on the highest weight vertex $\emptyset$ of $B(\lambda)$ with $f_j$ for $j\in J$. We use induction to check that if $F$ has the mentioned property, then so does $f_j(F)$. To see this, we recall the construction of $f_j(F)$ in Section~\ref{alcmod}, particularly \eqref{defw} and \eqref{eqn:rootF}. With this notation, we have $w(\beta_k)=\alpha_j$ for some $w$ in the chain \eqref{eqn:admissible}, where we know that $w\in W_J$ by the induction hypothesis. It follows that $\beta_k=w^{-1}(\alpha_j)$ lies in $\Phi_J$. Moreover, a similar reasoning shows that any $F$ with the mentioned property lies in $B(\lambda,J)$; more precisely, we apply to $F$ a sequence of crystal operators $e_j$ with $j\in J$, until we reach $\emptyset$.

It is clear that all the edges of $B(\lambda,J)$ are in $B_{w_\circ(J)}(\lambda)$. To show that the two crystals are isomorphic, we need to check that, for $j\not\in J$ and $F\in B_{w_\circ(J)}(\lambda)$, we cannot have $f_j(F)\in B_{w_\circ(J)}(\lambda)$. This is done based on the above description of the vertices of $B_{w_\circ(J)}(\lambda)$, and by using again, in a similar way, the description of the crystal operators in Section~\ref{alcmod}; in particular, the relation $w(\beta_k)=\alpha_j$ now implies $\beta_k\not\in\Phi_J$. 
\end{proof}

We will implicitly use the well-known fact that the sequence of roots corresponding to an alcove path from $A_\circ$ to $w_\circ(A_\circ)=-A_\circ$ (see Definition~\ref{deflc}) is a reflection order on $\Phi^+$; see~\cite[Remark~10.4~(2)]{LP1}. 

\begin{proof}[Proof of Theorem~{\rm \ref{unique-min}}]  Since the key is an element of the set $W^K$ of lowest coset representatives modulo $W_K$, it follows that $w_\circ(J)$ can be a key if and only if $J\subseteq K^c$. 

By Lemma~\ref{replem}, it suffices to assume that $\Phi$ is a finite root system, $\lambda$ is a regular weight, and $J=I$, so we can use the alcove model in the finite case. Indeed, the fiber of the key map at $w_\circ(J)$ is contained in $B_{w_\circ(J)}(\lambda)$ (cf. the description of the Demazure crystal, recalled in Section~\ref{sectkey}), and the crystal isomorphism in Lemma~\ref{replem} preserves keys (by Corollary~\ref{keyiso}). Since in the above special case $B(\lambda)$ has a maximum, whose key is $w_\circ=w_\circ(I)$, this is also the maximum in the fiber of the key map at $w_\circ$. Thus, it suffices to prove that the mentioned fiber has a minimum. 

It follows from definitions that $\lambda':=\lambda-\rho$ is a dominant weight. Based on the above discussion, we can choose an alcove path from $A_0:=A_\circ$ to $A_m:=A_\circ-\lambda$ which contains the alcoves $A_k:=A_\circ-\lambda'$ and $A_l:=w_\circ(A_\circ)-\lambda'$. The hyperplanes separating $A_k$ and $A_l$ are all the hyperplanes through $-\lambda'$ orthogonal to the roots in $\Phi^+$. Let $F_*:=\{k+1,k+2,\ldots,l\}$. Clearly, $F_*$ is an admissible subset and $\kappa(F_*)=w_\circ$. 

We will show that $F_*$ is the desired minimum. We can see that none of the hyperplanes separating $A_l$ and $A_m$ are orthogonal to a simple root; indeed, the sequence $(\beta_{k+1},\,\beta_{k+2},\,\ldots,\,\beta_m)$ contains each simple root $\alpha_i$ precisely once, as $\langle\rho,\alpha_i^\vee\rangle=1$, but there is such an occurrence in $(\beta_{k+1},\,\beta_{k+2},\,\ldots,\,\beta_l)$. Now consider an arbitrary admissible subset $F\ne F_*$ with $\kappa(F)=w_\circ$. By the above remark concerning the hyperplanes separating $A_l$ and $A_m$, we have $F\subseteq \{1,2,\ldots,l\}$, because the last element of $F$ must clearly be a simple root (recall Definition~\ref{defadm}~(1) of an admissible subset). As $F\ne F_*$, the set $\{k+1,k+2,\ldots,l\}\setminus F$ is non-empty, so let $t$ be its maximum. Let $F=\{j_1<j_2<\ldots<j_a<\ldots<j_s\}$ where $j_a<t<j_{a+1}$ (if $a=s$ then we drop $j_{a+1}$). Recall the definition \eqref{defw} of the root $\gamma_t$, namely
\[\gamma_t=r_{j_1}r_{j_2}\ldots r_{j_a}(\beta_t)=w_\circ r_{j_s}\ldots r_{j_{a+1}}(\beta_t)\,.\]
Due to the choice of $t$, Lemma~\ref{lemreford} applies, so $\gamma_t=-\alpha_p$, where $\alpha_p$ is a simple root. 

Now recall from Section~\ref{alcmod} the definition of the crystal operator $e_p$, which is based on the sequence $(\sigma_1,\ldots,\sigma_{n+1})$. We have $\langle\gamma_\infty,\alpha_p^\vee\rangle=\langle \rho,w_\circ(\alpha_p^\vee)\rangle<0$, since $w_\circ(\alpha_p)$ is a negative root, so $\sigma_{n+1}=-1$. Moreover, due to the choice of $t$ and the fact that $\gamma_t=-\alpha_p$, we have $\sigma_n=(1,-1),\,\ldots,\,\sigma_{q+1}=(1,-1),\,\sigma_q=(-1,-1)$, for some $q$. By property (C3) of the sequence $(\sigma_1,\ldots,\sigma_{n+1})$, we can now deduce $\sigma_{q-1,2}=-1$. By \eqref{eqn:graph_height} and \eqref{risedepth}, we see that $-\delta(F,p)\ge 2$, so we can apply $e_p$ at least twice to $F$. By \eqref{changekey}, we have $\kappa(e_p(F))=\kappa(F)=w_\circ(J)$, so $e_p(F)$ is still in the considered fiber. We can repeat the whole procedure above for $e_p(F)$ instead of $F$, and continue in this way; the process finishes (because the crystal has a minimum), and at that point it is clear that we reached $F_*$. We conclude that $F \ge 
 F_*$ in the crystal poset, as desired.
\end{proof}

\begin{remark}\label{fibers-lr-keys}{\rm 
In the case of a finite type crystal, the existence of the minimum in the fiber cannot be deduced from that of the maximum, via the Lusztig involution on the crystal (see Remark~\ref{labint}~(2)). Indeed, this involution does not map a fiber to another fiber. Instead, as Theorem~\ref{relkeys} shows, it maps a fiber of the (right) key map anti-isomorphically to one of the left key map, and viceversa. This implies that the fibers of the left key map at $\floor{w_\circ^J}$ also have a minimum and a maximum if $J\subseteq K^c$. 
}
\end{remark}
 
We will now give an example to show that Theorem~\ref{unique-min} fails for fibers of arbitrary Weyl group elements; in fact, a fiber can even be disconnected, as shown below. This justifies our focus on longest elements of parabolic subgroups.

\begin{example}{\rm Consider type $A_3$, with $\lambda=(3,2)$ and $w=2413$ in the symmetric group $S_4$. The fiber of the key map at $w$ has the structure shown in Figure~\ref{badfiber}, as an induced poset from the crystal poset $B(\lambda)$.

\begin{figure}[h]
    \centering
\[		
\begin{array}{cc}
\begin{diagram}
\node{\tableau{1&2&2\\2&4}}\\
\node{\tableau{1&1&2\\2&4}}\arrow{n,l}{1}
\end{diagram}
&\;\;\;\;\;\;\;\;\;\;\;
\begin{diagram}
\node[2]{\tableau{2&2&2\\4&4}}\\
\node[1]{\tableau{2&2&2\\3&4}}\arrow{ne,t}{3}
\node[2]{\tableau{1&2&2\\4&4}}\arrow{nw,t}{1}\\
\node[1]{\tableau{1&2&2\\3&4}}\arrow{n,l}{1}\arrow{nee,t}{3}
\node[2]{\tableau{1&1&2\\4&4}}\arrow{n,r}{1}\\
\node[2]{\tableau{1&1&2\\3&4}}\arrow{nw,b}{1}\arrow{ne,b}{3}
\end{diagram}
\end{array}\]
\caption{}
\label{badfiber}
\end{figure}
}
\end{example}
  
\section{M\"obius function and homotopy type for lower and upper intervals in a crystal poset}

 In this section,
    we will determine the M\"obius function for lower and upper intervals of a crystal as well as the homotopy type of the order complex of the proper part of each such interval.  These results will rely on 
 Theorem~\ref{unique-min}, namely  
  the result that each fiber  $\kappa^{-1}(w_\circ (J))$ has a unique smallest and unique largest element in a crystal. Unless otherwise specified, we continue to work in the symmetrizable Kac-Moody setup  
    at the beginning of Section ~\ref{key-fibers-section}, and we use the notation introduced there.

First  let us recall  a result that we will need that appears as Lemma 3.2.3 in  \cite{BB}.  It is important to note that this result does not require $W$ to be of finite type:
 
 \begin{theorem}[Lemma 3.2.3 of \cite{BB}]\label{longest-element-description}
 Let $J\subseteq S$.  Then $\vee_{j\in J} a_j$ exists if and only if $W_J$ is finite, in which case
 $\vee_{j\in J} a_j = w_\circ (J)$.  
Otherwise, there is no upper bound  for the elements of $J$.
 \end{theorem}
 
  In particular, this result directly  implies that  $w\in W$ is the longest element $w_\circ (J)$  of a finite parabolic subgroup of $W$ if and only if  each simple reflection appearing in a reduced expression for $w$ appears as the leftmost letter in a reduced expression for $w$ which in turn is true if and only if each such  simple reflection 
  also  appears as the rightmost letter in some reduced expression for $w$.   It may also be helpful to note that the closed interval $[\hat{0},w]$ in weak Bruhat order includes exactly the same elements as  the closed interval   $[\hat{0},w]$ in strong Bruhat order if and only if $w = w_\circ (J)$ for some parabolic subgroup  $W_J$. 
 
We will also need the following consequence of Lemma~3.2.3 from \cite{BB} that is closely related to ideas discussed in \cite{BB}. 
 
  \begin{corollary}\label{weak-lower-bound}
  For each element $w\in W$, there is a unique maximal element below $w$ in the weak Bruhat order among the longest elements $w_\circ(J)$ of parabolic subgroups.
  \end{corollary}
  
  \begin {proof}
  Simply note that $J$ must be exactly the simple reflections $s$ 
  corresponding to those atoms $a_s$ satisfying $a_s\le w$, since $w$ is an upper bound for the set of atoms below it. 
  \end{proof}
  
    In what follows, we will  also use the following result  that appears in \cite{Kas}, cf. also \cite[Section~1]{Li2}. This result also  follows  directly from the contrapositive of implication (\ref{firstkey}) described in Section~\ref{sectkey},
as explained in Remark~\ref{explain-contrapositive} below.

\begin{theorem}\label{all-reduced-words}
Consider any element $x$ with key $w = \kappa (x)$  in a crystal for  a symmetrizable Kac-Moody algebra. 
 Let $s_{i_1}\cdots s_{i_d}$ be any reduced expression for $w$.  Then there is a saturated chain proceeding downward from $x$ to $\hat{0}$ in the crystal as follows.    One first applies $e_{i_1}$ as many times as possible to $x$  while the result remains nonzero;
  then likewise one applies $e_{i_2}$ as many times as possible while staying nonzero, and one  continues in this manner from left to right through the reduced expression, thereby obtaining a saturated chain which begins at $x$ and proceeds downward to $\hat{0}$. 
\end{theorem}

The above string of crystal operators $e_i$ mapping $x$ to $\hat{0}$ is known as an {\em adapted string} (associated to the fixed reduced expression for $w$). 
 
\begin{remark}\label{explain-contrapositive}{\rm 
 The contrapositive of (\ref{firstkey}), which holds in the level of generality of symmetrizable Kac-Moody algebras,  ensures that each $e_{i_j}$ in turn may  be applied at least once in the setting of Theorem ~\ref{all-reduced-words}  while still yielding nonzero elements of the crystal.  Only the final step in any such $i_j$-string proceeding downward can change the key, but the fact that $w$ has length $d$ ensures that each such step does change the key.  Thus, the key must be the identity permutation after the last application of $e_{i_d}$, ensuring that the lowest element in the saturated chain is  $\hat{0}$.  }
\end{remark}

 Lemma~\ref{weak-lower-bound} together with Theorem ~\ref{all-reduced-words}  allows us to deduce:
  
  \begin{corollary}\label{matching-corollary}
  For each crystal vertex $x$, there is a unique maximal element $z \le x$ among the minima of  the fibers $\kappa^{-1} (w_\circ (J))$ for the various choices of $J\subseteq I$.
  \end{corollary}
  
  \begin{proof}
  First we calculate the key of $x$, namely $\kappa (x)\in W$.  
  Then we use Corollary~\ref{weak-lower-bound}
   to ensure the existence of a 
   unique maximal element  $w_\circ(J)$ satisfying $w_\circ (J) \le \kappa (x)$ in weak Bruhat order.   We will show that the minimal 
  element $z$ of the fiber $\kappa^{-1}(w_\circ (J))$  satisfies $z\le x$.  To this end, first we show that there exists an element $x'\in \kappa^{-1}(w_\circ (J))$ such that $x' \le x$, by virtue of having
  $w_\circ(J)\le \kappa (x)$ in weak Bruhat order, since it is known by Theorem~\ref{all-reduced-words} that we may obtain a saturated chain from  $\hat{0}$ to  $x$ by taking any reduced expression
  $s_{i_1}\cdots s_{i_d}$  for $\kappa (x)$ and first applying $e_{i_1}$ as many times sequentially  as possible to $x $ while staying nonzero, proceeding downward in the saturated chain, then likewise 
  applying $e_{i_2}$ as many times as possible, and continuing in this fashion through the reduced expression  from left to right;  we choose $s_{i_1}\cdots s_{i_d}$ so that a final 
  segment of it gives a reduced expression 
  $s_{i_r}\cdots s_{i_d}$  for $w_\circ(J)$, which means that just after applying the final copy of $e_{i_{r-1}}$ we will have reached a crystal  element $x'$  satisfying  $\kappa (x') = w_\circ(J)$, as desired.  In particular, this implies $x'\le x$.  But we proved for finite crystals in Theorem~\ref{unique-min}   that each  $w_\circ(J)$  has 
   a unique minimal element  $z$ with $\kappa (z) = w_\circ (J)$, implying  $z\le x' \le x$, giving the desired element as  $z$.
  \end{proof}

 
 \begin{theorem}\label{homotopy-proof}
 Every nonempty 
 open interval $(\hat{0},x)$   in a 
 crystal  
 $B(\lambda )$ corresponding to a symmetrizable Kac-Moody algebra 
 has order complex which is homotopy equivalent to a ball or a sphere.  Specifically, 
 $\Delta (\hat{0},x)\simeq 
 S^{|J|-2}$  if $x$ is the lowest 
    element in the fiber of the key map at $w_\circ(J)$ for some $J$ with 
 $|J|\ge 2$, while
 $\Delta (\hat{0},x)$ is contractible (or empty) otherwise.  
  \end{theorem}
 
 \begin{proof}
 
 Based on Corollary ~\ref{matching-corollary}, for each $y$ in the crystal, we let 
 $f(y)$ be the unique maximal element $z \le x$ among the minima of the fibers of the key map at the longest elements of parabolic subgroups $W_J$. 
   Thus,  $f$ gives a poset map to the subposet of minimal elements of fibers $\kappa^{-1}(w_\circ (J))$, for various subsets $J$ of the atoms of the crystal.  
 
  
  
Now the Quillen fiber lemma implies  that  
 $\Delta (\hat{0},x)$  is homotopy equivalent to the order complex of the subposet comprised of the minimal elements of the fibers of longest elements 
 $w_\circ (J)$  of parabolics for those $J$ which are subsets of the set of atoms below $x$, namely a Boolean algebra $B_{|J|}$ with $\hat{0}$ removed and possibly also with $\hat{1}$ removed.  More specifically, $\hat{1}$ is absent from this truncated Boolean algebra if and only if
 $x$ is itself the minimal element of a fiber of a longest element of a parabolic subgroup of $W$.    Finally, we use that the order complex of a Boolean algebra $B_d$  with both $\hat{0} $ and $\hat{1}$ removed  is homotopy equivalent to a sphere $S^{d-2}$, due to being the barycentric subdivision of the boundary of a $(d-1)$-simplex,  while the order complex for a Boolean algebra with $\hat{0}$ removed and with $\hat{1}$ present  is homotopy equivalent to a ball (due to having a cone point at $\hat{1}$).
 \end{proof}
 
 Recall that $\lfloor w \rfloor $ denotes the lowest representative of the coset $wW_{\lambda }=w W_K$ in our analysis below of the crystal $B(\lambda )$.

\begin{corollary}
Every  nonempty upper interval $(x,\hat{1})$ in a finite type crystal $B(\lambda )$ 
satisfies  $\Delta (x,\hat{1})\simeq 
 S^{|J|-2}$  if $x$ is 
 the highest element in the fiber of the left key map at $\floor{w_\circ^J}$ for some $J\subseteq K^c$ with $|J|\ge 2$, and it is contractible (or empty)  otherwise. 

\end{corollary}
\begin{proof}
This immediately follows via the Lusztig involution (see Remark~\ref{labint}~(2)), based on Theorem~\ref{relkeys} and Remark~\ref{fibers-lr-keys}.
\end{proof}

 By the well-known relationship $\mu_P(u,v) = \tilde{\chi}(\Delta (u,v))$ appearing e.g. as Proposition 3.8.6 in \cite{St}, Theorem ~\ref{homotopy-proof}  immediately yields: 

  \begin{corollary}\label{Moebius-theorem}
  For  
   crystals given by symmetrizable Kac-Moody algebras, 
    $\mu (\hat{0},x) = (-1)^{|J|-2}$ for $x$ the lowest 
    element in the fiber of the key map at $w_\circ(J)$; otherwise   
    $\mu (\hat{0},x)=0$. Likewise, in a finite type crystal, $\mu (x,\hat{1}) = (-1)^{|J|-2}$ for $x$ the 
    highest element in the fiber of the left key map at $\floor{w_\circ^J}$ for some $J\subseteq K^c$; otherwise, $\mu (x,\hat{1})=0$. 
 \end{corollary}

 \begin{remark}
{\rm 
Our approach to crystals above 
 is closely related to the viewpoint taken in \cite{BB} for the weak Bruhat  order, with our 
 poset map for crystals used in conjunction with the Quillen fiber lemma above specializing  
   to a poset map given in \cite{BB} for weak Bruhat order, and with 
 our usage of the Quillen fiber lemma accomplishing the same thing as the usage of properties of 
 (dual) closure maps in \cite{BB}. 
}
\end{remark}

 \begin{remark}{\rm 
 In the case of  non-finite type,  $\Delta (\hat{0},u) $ is nearly always contractible, implying that  
 $\mu (\hat{0},u)$ is nearly always $0$.   The point is that a set $J$ 
 of atoms will not have any upper bound unless $W_J$ is finite, meaning that very few elements in a non-finite type crystal will have key that is the longest element of a parabolic with the further property of minimality amongst elements with this key.
}
\end{remark}


We also  have the following corollary of Theorem~\ref{homotopy-proof}.

\begin{corollary}\label{crystal-type}
Given a rank $r\ge 2$ simple Lie algebra $\mathfrak{g}$ of finite type, the order complex of the proper part of the  crystal poset $B(\lambda)$ is homotopy equivalent to the sphere $S^{r-2}$ if $\lambda=\rho$, and to a ball otherwise.  
\end{corollary}

\begin{proof}
By Theorem~\ref{homotopy-proof}, the order complex of $B(\lambda)\setminus \{ \hat{0},\hat{1} \} $ is homotopy equivalent to a ball or a sphere, and we have a precise condition for when the latter happens, which we now assume. It is well known (see, e.g., \cite[Proposition~5.1]{lenccg}) that $\kappa(\hat{1})$ is the lowest  representative $w_\circ^K=w_\circ(I) w_\circ(K)$ of the coset $w_\circ W_\lambda$. We can assume that $K^c\ne\emptyset$, as $K=I$ would imply $\lambda=0$. We necessarily have $w_\circ^K=w_\circ(J)$ for some $J\subseteq K^c$. We claim that this implies $K=\emptyset$, so $\lambda$ is regular. Indeed, assuming the contrary and using the fact that $\ell(w_\circ(J))=|\Phi_J^+|$,  we derive the following contradiction from $w_\circ= w_\circ(J)w_\circ(K)$:
\[|\Phi^+|\le|\Phi_J^+|+|\Phi_K^+|\le |\Phi_K^+|+|\Phi_{K^c}^+|<|\Phi^+|\,.\]
Here the last strict inequality is justified as follows. The corresponding weak inequality is clear since $\Phi_K^+$ and $\Phi_{K^c}^+$ are disjoint, so assume that they partition $\Phi^+$. It follows that each $s_\alpha$ for $\alpha$ in $K$ (resp. $K^c$) permutes $\Phi_{K^c}$ (resp. $\Phi_K$), so $\alpha$ is orthogonal to all the roots of $\Phi_{K^c}$ (resp. $\Phi_K$). But this is impossible since $\Phi$ is an irrreducible root system. 

Under the assumption at the beginning, we proved that $\langle\lambda,\alpha_i^\vee\rangle\ge 1$ for each simple root $\alpha_i$. It remains to prove that in each case we have equality, which means that $\lambda=\rho$. Indeed, assume that $\langle\lambda,\alpha_i^\vee\rangle\ge 2$ for some $\alpha_i$. From Theorem~\ref{homotopy-proof} we know that $\hat{0}$ is the highest element in the fiber of the left key map at the identity in $W$. We claim that $\overline{\kappa}(f_i(\hat{0}))=1$, which contradicts the previous fact. Indeed, consider one of the special $\lambda$-chains in Definition~\ref{defkeys} (see the related notation), which contains the root $\alpha_i$ in positions $\overline{\jmath}<j_1<\ldots<j_s$, with $s=\langle\lambda,\alpha_i^\vee\rangle-1\ge 1$. By the definition of the action of $f_i$ in \eqref{eqn:rootF}, we have $f_i(\hat{0})=f_i(\emptyset)=\{j_s\}$, which makes the claim obvious once we recall Definition~\ref{defkeys}.  

Conversely, the order complex of $B(\rho)\setminus \{ \hat{0},\hat{1} \} $ is homotopy equivalent to a sphere since
\[\overline{\kappa}(\hat{0})=1=w_\circ^I\,,\;\;\;\;\overline{\kappa}(f_i(\hat{0}))=s_i>1\;\;\mbox{for all $i\in I$}\,.\]
The second property is proved in the same way as the similar one above.
\end{proof}

Let us comment on the meaning of Corollary~\ref{crystal-type}. It shows that the crystals $B(\rho)$ are special, in the sense that they behave, up to homotopy, in the same way as the weak Bruhat order on the corresponding Weyl group. This is not the case for $B(\lambda)$ for regular $\lambda\ne\rho$, which can be thought of as ``fattened'' versions of $B(\rho)$. It is also not the case for $B(\lambda)$ for non-regular $\lambda$, which can be thought of as degenerate versions of $B(\lambda)$ for regular $\lambda$.

 \section{Negative results for arbitrary intervals in  (type $A$) crystal posets}\label{negative-section}
 
 This section constructs explicit examples in type $A$ to exhibit intervals with arbitrarily large 
 M\"obius function, as well as open intervals of arbitrarily large rank that are disconnected.  In both cases, we use the SSYT model to describe our constructions.

The starting point for obtaining  all of  our 
counterexamples is the interval shown in Figure~\ref{badint}, in the type $A$ crystal poset corresponding to the partition $(4,3)$. Note that the M\"obius function of this interval is $2$.  We will construct two different  infinite families of examples, each  of  which has this example as its based case.  
 
\begin{figure}[h]
    \centering
\[\begin{diagram}
\node[4]{\tableau{1&1&2&3\\3&4&4}}\\
\node[2]{\tableau{1&1&2&3\\3&3&4}}\arrow{nee,t}{3}
\node[2]{\tableau{1&1&2&2\\3&4&4}}\arrow{n,r}{2}
\node[2]{\tableau{1&1&1&3\\3&4&4}}\arrow{nww,t}{1}\\
\node[2]{\tableau{1&1&2&3\\2&3&4}}\arrow{n,l}{2}
\node[1]{\tableau{1&1&1&2\\3&4&4}}\arrow{ne,t}{1}
\node[2]{\tableau{1&1&2&2\\3&3&4}}\arrow{nw,t}{3}
\node[1]{\tableau{1&1&1&3\\2&4&4}}\arrow{n,r}{2}\\
\node[2]{\tableau{1&1&2&2\\2&3&4}}\arrow{n,l}{2}
\node[2]{\tableau{1&1&1&2\\3&3&4}}\arrow{nw,b}{3}\arrow{ne,b}{1}
\node[2]{\tableau{1&1&1&2\\2&4&4}}\arrow{n,r}{2}\\
\node[4]{\tableau{1&1&1&2\\2&3&4}}\arrow{nww,b}{1}\arrow{n,r}{2}\arrow{nee,r}{3}
\end{diagram}\]
\caption{}
\label{badint}
\end{figure}

 \begin{theorem}\label{maxchains}
 No finite set of moves suffices to connect  the sets of maximal chains in all closed  intervals
 $[u,v]$  of  type $A$ crystal posets.  In particular, there are disconnected
 open intervals $(u,v)$ of arbitrarily large rank.  
 \end{theorem}
 
 \begin{proof}
We proceed by constructing a series of examples of disconnected open intervals $(u,v)$ of arbitrarily large rank within type $A$ crystal posets. More precisely, for each rank $n\ge 3$, we consider the two-row partition $(n+1,n)$, and the interval between the two SSYT below, cf.~Figure~\ref{badint}:
\cellsize=5.3ex
\[u:=\tableau{1&1&1&2&{\cdots}&{n\!-\!2}&{n\!-\!1}\\2&3&4&5&{\cdots}&{n\!+\!1}}\,,\;\;\;
v:=\tableau{1&1&2&{\cdots}&{n\!-\!2}&{n\!-\!1}&{n}\\3&4&5&{\cdots}&{n\!+\!1}&{n\!+\!1}}\,.\]

Note that $v$ is obtained from $u$ by incrementing by $1$ the entries $1,2,\ldots,n-1$ in row $1$ (we refer to the rightmost entry $1$), as well as the entries $2,3,\ldots,n$ in row $2$. Also note that there is a label increasing and a label decreasing chain from $u$ to $v$, with the following sequences of labels:
\[(1,2,2,3,3,\ldots,n-1,n-1,n)\,,\;\;\;(n,n-1,n-1,\ldots,3,3,2,2,1)\,.\]
Indeed, $f_1$ on $u$ increments the rightmost entry $1$ in row $1$, $f_2^2$ increments the entries $2$ in rows $1$ and $2$ of $u$, $\ldots$, $f_{n-1}^2$ increments the entries $n-1$ in rows $1$ and $2$ of $u$, and finally $f_n$ increments the entry $n$ in row $2$ of $u$; in each case, there is a single $-+$ pair (see Section~\ref{ssyt}). Similarly for the weakly decreasing chain.

 To prove for each $n\ge 3$ that $(u,v)$ is disconnected, 
  we show that the connected component containing the saturated chain with labels $(1,2,2,3,3,\ldots,n-1,n-1,n)$ only contains saturated chains starting with the label 1 and ending with the label $n$.  In particular, this connected component  does not include the saturated chain with labels $(n,n-1,n-1,\ldots,3,3,2,2,1)$. It suffices to prove  two things:
  (1) that there do not exist saturated chains starting with a label $k\ne 1$ and ending with the label $n$; in fact, we 
  can also assume $k\ne n$, cf. Remark~\ref{labint}~(1)), in this part; and (2) that  there do not exist saturated  chains starting with the label $1$ and ending with a label $k\not\in\{1,n\}$.
	
	 Let us now turn to the first case, carrying this out in detail for $k=2$ and explaining how this can be modified to handle any $k$ with $1<k < n$, since we already ruled out $k=n$.
	 It will suffice in this $k=2$  case  to prove  
	that the following two elements, denoted $u',v'$, are incomparable: 
\[\tableau{1&1&1&2&{\cdots}&{n\!-\!4}&{n\!-\!3}&{n\!-\!2}&{n\!-\!1}\\3&3&4&5&{\cdots}&{n\!-\!1}&{n}&{n\!+\!1}}\,,\;\;\;
\tableau{1&1&2&{\cdots}&{n\!-\!3}&{n\!-\!2}&{n\!-\!1}&{n}\\3&4&5&{\cdots}&{n}&{n}&{n\!+\!1}}\,.\]
\cellsize=2.5ex
To generalize this  case to larger $k$ satisfying $2<k<n$, the  point will be to replace $u'$ by a tableau obtained from $u$ by applying the operator $f_k$ in place of $f_2$, which was applied to $u$ to obtain $u' = f_2(u)$ above.  Then one may easily check that the reasoning below will still apply to show that $f_k(u)$ is also incomparable to $v'$ above.  

Throughout the following argument, we denote by $i_j$ the rightmost entry $i$ in row $j$ of $u'$. 
Suppose by way of  contradiction the existence of a saturated chain from $u'$ to $v'$, which means that the entries $1_1,2_1,\ldots,(n-1)_1$ and $3_2,4_2,\ldots,(n-1)_2$ are incremented by $1$, in some order. We write $i_j\prec i'_{j'}$ if $i_j$ is incremented before $i'_{j'}$. 

We first observe that we have
\begin{equation}\label{precrel}
(n-2)_1\prec (n-1)_1\,,\;\;\mbox{and}\;\;(\mbox{$(i-1)_1\prec i_1$ or $(i+1)_2\prec i_1$, for $i=3,4,\ldots,n-2$})\,.
\end{equation}
Indeed, if neither $(i-1)_1$ nor $(i+1)_2$ were incremented when $i_1$ is considered, then $(i+1)_2$ is paired with $i_1$ at this moment, so the latter cannot be incremented, as it is supposed to be. 
Now consider \eqref{precrel} for $i=n-2$.  Suppose that $(n-1)_2\prec(n-2)_1$, which implies, based on the first precedence in \eqref{precrel}, that $(n-1)_2\prec (n-2)_1\prec (n-1)_1$; but then, when $(n-1)_1$ is considered, this entry and the incremented $(n-2)_1$ to its immediate left are paired with the incremented $(n-1)_2$ and the not-to-be-changed $n_2$, respectively, which makes it impossible for $(n-1)_1$ to be incremented.  This gives a contradiction to our assumption that
$(n-1)_2\prec (n-2)_1$. 

 Recalling \eqref{precrel} for $i=n-2$, this implies that 
\begin{equation}\label{precrel1}
(n-3)_1\prec (n-2)_1\prec (n-1)_2\,.
\end{equation}
Now consider \eqref{precrel} for $i=n-3$ and assume that $(n-2)_2\prec(n-3)_1$, which implies, based on \eqref{precrel1}, that $(n-2)_2\prec(n-3)_1\prec(n-2)_1\prec (n-1)_2$; but then, when $(n-2)_1$ is considered, this entry and the incremented $(n-3)_1$ to its left are paired with the incremented $(n-2)_2$ and the not-yet-incremented $(n-1)_2$, respectively, which makes it impossible for $(n-2)_1$ to be incremented. Recalling \eqref{precrel} for $i=n-3$, this implies that 
\begin{equation*}
(n-4)_1\prec (n-3)_1\prec (n-2)_2\,.
\end{equation*}
Now consider \eqref{precrel} for $i=n-4$, repeating the same argument, { then continue in this manner for $i=n-j$ for each successively larger $j$ in turn}.  
Continuing in this way, we eventually deduce that $2_1\prec 3_2$. This means that, when $2_1$ is considered, it is paired with $3_2$ or the $3$ to its left, so it cannot be incremented. This concludes the proof by contradiction in the first case.

Now let us indicate how the second case  follows by a similar chain of implications.  Let $u' = f_1(u)$ and $v'=e_k(v)$, for any $k$ satisfying $1<k<n$.  We will show that there is no saturated chain from $u$ to $v$ which includes both $u'$ and $v'$, in the case $k=n-1$, by an argument that will work equally well for each $1<k<n$.  This will involve deducing a series of relations $i_j \prec k_l$ similar to the first case.   Now let us provide the series of such relations, 
generally leaving it to the reader to justify the corresponding implications.   By the definition of $u'$, we have $1_1 \prec  2_2$, which then  implies $2_1 \prec 2_2$.  But this implies $2_2 \prec 3_2$ (hence by transitivity also implies $2_1 \prec 3_2$).  But then  $2_1\prec 3_2$  implies  $3_1\prec 3_2$ and $3_2\prec 4_2$ (which by transitivity also  yields  $3_1\prec 4_2$). Continuing in this fashion, we get $i_1\prec i_2$ for each $i$ satisfying $1 < i < n$.  But this contradicts the assumption that $v' = e_{n-1}(v)$,  since  applying  $e_{n-1}$ to $v$ will decrement the value $n$ from the first row.  This completes the proof.
 \end{proof}
 
\begin{remark}{\rm Using the setup and the method in the proof of Theorem~\ref{maxchains}, one can give an explicit description of all the saturated  chains from $u$ to $v$ in the connected component of the one labeled $(1,2,2,3,3,\ldots,n-1,n-1,n)$. Beside the fact that the first label in such a chain is $1$ and the last one is $n$ (as shown above), this description is best given in terms of the precedence relations on the pairs $i_j$, with $i=2,\ldots,n-1$ and $j=1,2$, specifying an entry $i$ in row $j$ of $u$ to be incremented. The poset formed by the necessary and sufficient precedence relations is the Cartesian product of the chain $2\prec 3\prec\ldots\prec n-1$ corresponding to the values of $i$ and the chain $1\prec 2$ corresponding to the values of $j$. The saturated chains from $u$ to $v$ are in bijection with the linear extensions of this poset, so they are counted by the Catalan number $C_{n-2}$. 
}
\end{remark}

 \begin{theorem}\label{big-Moebius}
 There are type $A$ crystal poset intervals
  with arbitrarily large M\"obius function. 
 \end{theorem}
 
 \begin{proof}
 The plan is to take the interval $(u,v)$ in~Figure~\ref{badint} with $\mu (u,v) = 2$,  where the SSYT corresponding to $u,v$ are denoted $T_1$ and $T_2$,  and then construct new SSYT $T_1^{(r)}$ and $T_2^{(r)}$ so that 
 $\mu (T_1^{(r)}, T_2^{(r)}) =  2^r$.   To this end, it will suffice (by Proposition 3.8.2 in \cite{St})  to show that 
 $[T_1^{(r)},T_2^{(r)}]$ is isomorphic to an $r$-fold product of closed intervals each isomorphic to 
 $[T_1,T_2]$.
 
 
 The shape $\lambda^{(r)}$  for $T_1^{(r)}$ and $T_2^{(r)}$ will have $2r $ rows and $4r$ columns.  Specifically, $\lambda^{(r)}$  is the smallest shape that contains the skew shape $\nu^{(r)}  = \lambda^{(r)}  \setminus \rho^{(r)}  $ obtained by taking $r$ copies of $\lambda $ arranged so that the lower left corner of one copy of $\lambda $  touches the upper right corner of the next copy of $\lambda $ at a single point, with the  $r$ copies of $\lambda $ thus arranged from top right to lower left.  Now both $T_1^{(r)}$ and $T_2^{(r)}$ will have the letter $i$ in  each  box in row $i$ that belongs to the shape  $\rho^{(r)}$ that is deleted from the skew shape.   Thus, $T_1^{(r)}$ will coincide with $T_2^{(r)}$ in all these entries.  On the other hand, the upper right copy of $\lambda $ gets filled in with the same values in $T_1^{(r)}$ that it has in $T_1$ and with the same values in $T_2^{(r)}$ that it gets in $T_2$.  Now each time we move down from one copy of $\lambda $ to the next, as we proceed from top to bottom and from right to left through our skew shape,  we fill the new  copies  of $\lambda $ with the same values as those in the previous copy of $\lambda $, except that we add a fixed, sufficiently  large constant (e.g. 100)  to all the entries, in the new copy of $\lambda $ in comparison to their values in the previous copy of $\lambda $.   This will ensure 
 that  the interval $[T_1^{(r)},T_2^{(r)}]$ is isomorphic to the $r$-fold product 
 $[T_1,T_2]\times \cdots \times [T_1,T_2]$  with one copy of $[T_1,T_2]$ coming from each copy of the shape $\lambda $ along the lower right boundary of the shapes for $T_1^{(r)}$ and $T_2^{(r)}$.
 %
  %
 \end{proof}
 
 \begin{example}\label{non-lattice}
  {\rm Next we exhibit an example to show that the crystal posets are not always lattices.  Consider the crystal of SSYT of shape $(4,3)$ filled with $\{1,2,3,4\}$, and its elements
\[T=\tableau{{1}&{1}&{2}&{2}\\{2}&{3}&{4}}\,,\;\;\;\;S=\tableau{{1}&{1}&{1}&{2}\\{3}&{3}&{4}}\,,\;\;\;\;\mbox{covering }\:\tableau{{1}&{1}&{1}&{2}\\{2}&{3}&{4}}\,.\]
Now recall the structure of the interval $\left[\,\tableau{{1}&{1}&{1}&{2}\\{2}&{3}&{4}}\,,\;\tableau{{1}&{1}&{2}&{3}\\{3}&{4}&{4}}\,\right]$ in~Figure~\ref{badint}. We can see that $\tableau{{1}&{1}&{2}&{3}\\{3}&{4}&{4}}$ is a minimal upper bound for $T$ and $S$. However, we have the incomparable (also minimal) upper bound $\tableau{{1}&{2}&{2}&{3}\\{3}&{3}&{4}}=f_1 f_2^2(T)=f_2 f_1^2(S)$. Thus, the many techniques available for studying lattices are not at our disposal.} 

  \end{example}
 
 \color{black}
\section{M\"obius function  not $0,\pm 1$ guarantees relations amongst  crystal operators not implied by Stembridge's local relations} 
\color{black}
\label{SB-labeling}

Let us conclude with a result that will  tie  together our results, both positive and negative, from earlier sections.  Specifically, we will now make precise how  the M\"obius function may  be used, and indeed has been used, to discover new relations amongst crystal operators not implied by those giving rise to Stembridge's local structure for crystals of highest weight representations in the simply laced case.

First recall the following notions and results from \cite{HM} that we will rely upon.

\begin{definition}\label{lattice-labeling} 
{\rm 
An edge-labeling $\lambda $ of a finite lattice $L$ is a {\it lower $SB$-labeling} if it may be constructed as follows.  Begin with a label set $S$ such that there is a subset  $\{ \lambda_a | a\in A(L) \}  $ of $S$  whose members are in bijection with the set $A(L)$ of atoms of  $L$.
\begin{enumerate}
\item
No two labels upward from $\hat{0}$ to  distinct atoms may  be equal. 
This allows us to define  the label $\lambda_a $  on each   cover relation $\hat{0}\prec a$ as  the label corresponding to the  atom 
$a$.  
\item
Given any interval of the form 
$[\hat{0}, a_{i_1}\vee \cdots \vee a_{i_r} ] $ for  $ \{  a_{i_1},\dots
,a_{i_r} \}  \subseteq A(L)$, each of the saturated chains $M$ from $\hat{0} $ to $a_{i_1}\vee \cdots \vee
a_{i_r} $  has the property
that the set $\lambda (M)$ of labels occurring with positive multiplicity on $M$  is exactly 
$\{ \lambda_{a_{i_j}} | a_{i_j} \in A(L); 1\le j \le r \} $.
\end{enumerate}
When an edge labeling $\lambda $ for a finite lattice $L$  meets  these conditions upon restriction to each  closed interval of $L$, then we call such a labeling an {\it $SB$-labeling}.  We call a lattice  with an $SB$-labeling an {\it $SB$-lattice}. 
}
\end{definition}



\begin{definition}[\cite{HM}] 
{\rm 
 The {\it index $2$ formulation of   $SB$-labeling}  is an edge labeling on a  finite lattice $L$ satisfying the following conditions for each  $u,v,w \in L$ such that 
 $v$ and $w$ are distinct elements which each  cover $u$:
\begin{enumerate}
\item
$\lambda (u,v)\ne \lambda (u,w)$
\item
Each saturated chain on the interval $[u,v\vee w]$ uses both of these labels $\lambda (u,v)$ and $\lambda (u,w)$ a positive number of times.
\item
None of the saturated chains on the interval $[u,v\vee w]$ use any other types of labels besides $\lambda (u,v)$ and $\lambda (u,w)$.
\end{enumerate}
}
\end{definition}

Now let us recall the properties of $SB$-labelings that we will use.

\begin{theorem}[\cite{HM}]
An edge labeling on a finite lattice is an $SB$-labeling if and only if it satisfies the index $2$ formulation for $SB$-labeling.
\end{theorem}

\begin{theorem}[\cite{HM}]
Any finite lattice $L$ with an $SB$-labeling has the property that $\mu  (u,v)\in \{ 0,\pm 1\} $ for all 
$u\le v$ in $L$.
\end{theorem}

Next let us establish a new piece of terminology.

\begin{definition} 
{\rm 
Given distinct elements $x,y,z$ in a finite crystal graph given by a highest weight representation of simply laced type with $y = f_i (x)$ and $z = f_j (x)$, say that an upper bound $w$  for $y$ and $z$ is a {\it Stembridge local upper bound} if either  of the following conditions holds:
\begin{enumerate} 
\item 
  $w = f_j (y) = f_i (z)$ 
  \item 
$w = f_i f_j f_j (y) = f_j f_i f_i (z)$ and  $f_i f_j (x)\ne f_j f_i (x)$.
\end{enumerate}
Likewise, we call the consequent relations $f_if_j (x) = f_j f_i (x) $(in the first case) and $f_i f_j f_j f_i (x) = f_j f_i f_i f_j (x)$ (in the second case) among crystal operators  {\it Stembridge local relations}.  
}
\end{definition}

The critical observation behind  the main result of this section, Theorem ~\ref{moebius-yields-crystal-relations}, is the following implication.
Suppose that all crystals of highest weight representations  of finite,  simply laced type in fact 
 were lattices with the further property  that each triple $u,v,w$ of distinct elements with $v$ and $w$ both covering $u$ has unique least upper bound for $v$ and $w$ being a Stembridge local upper bound.  Then 
 the standard labeling of crystal graph edges by colors, namely the labeling with $u\prec f_i (u)$ colored $i$,  would be an $SB$-labeling for these lattices.
 This is explained within the proof of Theorem ~\ref{moebius-yields-crystal-relations}.

\begin{theorem}\label{moebius-yields-crystal-relations}
Given any $u<v$ in a crystal $L$ of a highest weight representation of finite, simply laced type such that 
$\mu (u,v)\not\in \{ 0, 1, -1\} $, then 
there must be a relation  within $[u,v]$  among crystal operators that is  not implied by the Stembridge local relations. 
\end{theorem}

\begin{proof}
First notice that the obvious labeling would satisfy the requirements for the index 2 formulation of $SB$-labeling if the interval $[u,v]$ were a lattice with the further property (SUB) of  $a\prec b$ and $a\prec c$ for $b\ne c$ implying that the unique least upper bound for $b$ and $c$ is a Stembridge local upper bound. Lemma ~\ref{local-least-upper-bounds} would give this property (SUB) in the event that  $L$ as a whole is a lattice.  
 The existence of an $SB$-labeling on a finite lattice implies that  $\mu (u,v)\in \{ 0,1,-1\} $  for each $u\le v$ in the lattice.   Thus, our hypothesis about the M\"obius function not taking one of these values  ensures that  either the lattice property fails within $[u,v]$ or that the unique least upper bound for  some such $b$ and $c$ within $[u,v]$ is something other than the Stembridge local upper bound.

 Recall e.g. from \cite{St}  that any finite  meet semi-lattice with a unique maximal element is a lattice, ensuring for $[u,v]$ not a lattice that there must be a pair of elements with two different least upper bounds.  
 Assuming the lattice property fails within $[u,v]$, we 
will prove the existence of a pair of elements $u' < v'$ within $[u,v]$ such that there are distinct elements $a_1,a_2$ both covering $u'$ such that $a_1$ and $a_2$ do not have a least upper bound.   Among all pairs $u' < v'$ within $[u,v]$, choose one having as few elements as possible within $[u',v']$ with the further property that there exists a pair of  elements $x,y\in [u',v']$ that do not have a unique least upper bound within $[u',v']$ but also requiring that there does not exist such a pair $x,y$ which both cover $u'$.  Let $m =  |(u',x]|$ and let $n= |(u',y]|$.  Choose $x$ and $y$ within $[u',v']$ so as to minimize $m+n$ among pairs that do not have a unique least upper bound within $[u',v]$.  Let $x', y'$ be this pair, and let $m_1$ and $m_2$ be two different least upper bounds for $x'$ and $y'$ within $[u',v']$.


   Suppose that  $m+n>2$. 
   This allows us to  assume  without loss of generality that  there also exists $x'' \in (u',v']$ such that $x''$ is covered by $x'$.  Notice that  $m_1$ and $m_2$ for $x'$ and $y'$ are also distinct, incomparable  upper bounds for $x''$ and $y'$.  By  our minimality assumption for $m+n$ and the fact that $|(u',x'']| < |(u',x']|$,  there must  be a unique least upper bound $m_3$ for $x''$ and $y'$.  This implies  $m_3 < m_1$ and $m_3 < m_2$.  But then $x'\not\le m_3$ since otherwise $m_1$ and $m_2$ could not be least upper bounds for $x'$ and $y'$.  On the other hand, $m_3\not\le x'$ since we do have $y'\le m_3$ and also have $u\not\le x'$.  But then $x'$ and $m_3$ are distinct incomparable elements of the interval $[x'',y']$ which have two different, incomparable  least upper bounds within $[x'',y]$, namely $x'$ and $m_3$.  But $[x'',y']$ has strictly fewer elements than $[u',v']$ contradicting our inductive hypothesis that  $|[u',v']|$ is  as small as possible.  The upshot is that there must be a pair of distinct elements  $x',y'$ within $[u,v]$ that both cover a common element $u' $ and where $x'$ and $y'$ do not have a unique least upper bound within $[u,v]$. This guarantees a relation that is not implied by the Stembridge local relations. 
\end{proof}

\begin{lemma}\label{local-least-upper-bounds}
Let $L$ be any poset whose Hasse diagram is  a crystal  graph given by a highest weight representation of  simply laced type.   
For any three distinct elements $x, f_i(x)$ and $f_j(x)$ of $L$, the Stembridge local upper bound $y$ for $f_i(x)$ and $f_j(x)$ 
 is a least upper bound for $f_i (x)$ and $f_j(x)$.  In other words, there does not exist any element $y'$ satisfying $x < y' < y$.
\end{lemma}

\begin{proof}
The case with $y = f_i f_j (x) = f_j  f_i (x)$ follows immediately from the distinctness of $f_i (x)$ and $f_j (x)$.  Now suppose $f_i f_j (x) \ne f_j f_i (x)$, which necessarily implies 
$y = f_i f_j f_j f_i (x) = f_j f_i f_i f_j (x)$.   We may then use weight considerations to see that all cover relations within the interval $[x,y]$ are colored either $i$ or $j$, implying there are at most two cover relations upward from any given poset element staying in $[x,y]$  and likewise at most two cover relations downward from any element staying within $[x,y]$.  Without loss of generality, we only need to consider as potential  least upper bounds for $f_i (x)$ and $f_j (x)$ within $[x,y]$ the elements (a) $f_i f_i (x)$,  (b) $f_i f_j f_i (x)$,   (c)  $f_j f_i f_i (x)$, and  (d) $f_j f_j f_i(x)$.  We rule out (a) by weight considerations, which ensure $f_j (x)\not \le f_i f_i (x)$.  Notice that (b) would imply $f_j (x) \le f_i f_j f_i (x)$,  which by weight considerations  forces 
$y' = f_i f_j f_i (x)  = f_i f_i f_j (x)$;  the fact that there is only one cover relation downward from $y'$ colored $i$ then implies $f_j f_i (x) = f_i f_j (x)$, contradicting our assumption that these are not equal.  
This leaves (c) and (d), which will turn out to be equivalent cases to each other, to be handled next.

Suppose we are in case (c), namely  suppose we have  
least upper bound  $y' = f_j f_i f_i (x)$ for $f_i (x)$ and $f_j (x)$.  Then weight considerations together with $f_j (x) \le f_j f_i f_i (x)$ force the relation $y' = x_j f_i f_i (x) = f_i f_i f_j (x)$, also putting us in  case (d).  But then $e_i (y')\ne 0$ and $e_j (y')\ne 0$ since these equal $f_i f_j (x)$ and $f_i f_i (x)$, respectively.  We cannot have $e_i e_j (y') = e_j e_i (y')$, since one may see immediately by drawing the diagram that  this would imply $f_i f_j (x) = f_j f_i (x)$, contradicting our hypothesis that these are not equal.
  Thus, Stembridge's local structure applied to the operators $e_i, e_j$  forces  there to be a nonzero  element we denote by 
$x'$  satisfying $x' = e_i e_j e_j e_i (y') = e_j e_i e_i e_j (y')$.  But then $x' = e_j (x) $ implying
$f_i f_i f_j f_j (x') = f_i f_i f_j (x) = f_j f_i f_i (x) = f_j f_i f_i f_j (x') = y' =  f_i f_j f_j f_i (x')$.  The fact that there is only one cover relation downward from $y'$ colored $i$ means that the  relation 
$f_i f_i f_j f_j (x') = f_i f_j f_j f_i (x') \ne 0$ we just deduced  also implies  $f_i f_j f_j (x') = f_j f_j f_i (x')\ne 0$.  Moreover, we must also have  $f_j f_i (x') \ne f_i f_j (x')$,  since equality would imply  $f_i f_j (x) = f_j f_i (x)$, as is seen by drawing the diagram; however, we assumed in  our inductive  hypotheses  that these are not equal.  Thus, the existence of $x$ having $f_i f_j (x)  \ne f_j f_i (x) $ with $f_i f_j (x) \ne 0$ and $f_j f_i (x)\ne 0$ and also having $f_j f_i f_i (x) = f_i f_i f_j (x)\ne 0$ implies that $x$ covers an element $x' = e_j (x) \ne 0$ where $x'$ satisfies this exact same list of hypotheses that $x$ satisfies, but with the roles of $i$ and $j$ reversed.  In particular, we may keep repeating this argument to obtain an infinite chain $x, x', x'',\dots $ of nonzero elements by letting 
$x' = e_j (x), x'' = e_i (x'), x''' = e_j (x''), $ and continuing in this fashion to alternate applying $e_i $ and $e_j$.
The existence of this infinite chain contradicts the existence of a unique minimal element $\hat{0}$ in  our crystal graph.  But any highest weight representation is generated by a single highest weight vector giving rise to a unique minimal element $\hat{0}$ in the crystal,  completing the proof.
\end{proof}

\end{document}